\documentclass[a4paper,reqno,11pt]{amsart}
\usepackage{amssymb}           
\usepackage{hyperref}
\usepackage{eucal}
\usepackage{a4wide}
\usepackage[svgnames]{pstricks}
\usepackage{pst-plot}
\usepackage{mathptmx} 

\numberwithin{equation}{section}

\usepackage[normalem]{ulem}

\hypersetup{colorlinks=true,linkcolor=violet,citecolor=purple}

{\theoremstyle{definition}
\newtheorem{definition}{Definition}[section]
\newtheorem{remarkth}[definition]{Remark}
\newtheorem{remarkths}[definition]{Remarks}

}
\newtheorem{lemma}[definition]{Lemma}
\newtheorem{theorem}[definition]{Theorem}
\newtheorem{proposition}[definition]{Proposition}
\newtheorem{corollary}[definition]{Corollary}
\renewcommand{\emph}[1]{{\bfseries\itshape{#1}}}

\newcommand{\R}{\mathbb{R}}      

\newcommand{\I}{I\mkern-7muI} 

\newcommand{\g}{\mathfrak{g}}

 \def\Mbar{\bar{M}}
 \def\eps{\varepsilon}
 \def\xx{\mathbf{x}}

\begin{document}

\title[ Gauge momenta  as Casimir functions of nonholonomic systems]{Gauge momenta  as Casimir functions of nonholonomic systems}

\author{Luis Garc\'ia-Naranjo \& James Montaldi}
\thanks{This research was made possible by a Newton Advanced Fellowship from the Royal Society.}
\email{ luis@mym.iimas.unam.mx,\quad j.montaldi@manchester.ac.uk}
\address{LGN: IIMAS, UNAM, Mexico City, Mexico. \quad JM: School of Mathematics, University of Manchester, UK.}
\date{\today}

\begin{abstract}
We consider nonholonomic systems with symmetry possessing a certain type of first integrals that are linear in the velocities. We develop a systematic method for modifying the standard nonholonomic almost Poisson structure that describes the dynamics so that these integrals become Casimir functions after reduction. This explains a number of recent results on Hamiltonization of nonholonomic systems, and has consequences for the study of relative equilibria in such systems.
\end{abstract}

\maketitle

\section{Introduction}

The search for geometric structures that are invariant by the flow of nonholonomic mechanical systems has driven  a good part of the research in nonholonomic systems from the time of
Chaplygin to the present day. These range from first integrals and invariant measures, to symplectic or Poisson structures.

The difficulty of the problem arises because, in the nonholonomic case,  the addition of the D'Alembert constraint reaction forces  destroys the Hamiltonian nature of the mechanical equations of motion.
 The resulting equations  preserve the energy of the system but only allow a formulation in terms of  {\em almost} Poisson structures \cite{vdSM,Marle,Cantrijn99,Ibort} which fail to satisfy the
Jacobi identity and are not preserved by the flow.

A look at  concrete nonholonomic examples (see e.g.\ the tables in \cite{BorMam,BorMam2} for a good overview), suggests that the presence of symmetries could lead to the existence of some geometric invariants in concrete problems, but  precise results in this direction are missing.

In recent years, much effort has  been devoted to the generalisation of Noether's Theorem (which in general terms is a statement linking symmetries to first integrals) to the nonholonomic setting (see for example \cite{Agostinelli, Iliev2,Fasso1, FassoJGM}).  
Of particular interest is the case in which the symmetries of the system arise as the lift of the action of a  Lie group $G$ on the configuration space. 
 In analogy with the Hamiltonian counterpart of this situation, one would expect that  if such first  integrals exist, they are linear on the velocities and remain constants of motion in the presence of $G$-invariant force potentials.
First integrals for nonholonomic systems having these properties are known in the literature as  \emph{horizontal gauge momenta} \cite{Bates,Fasso2,Fasso3}. In this paper we
do not consider other types of linear first integrals, so we will simply call them   \emph{gauge momenta}.  

A separate line of research in the field of nonholonomic systems with symmetry has focused on the geometry of their reduction  (see \cite{Koi,BS93,BKMM,CDS-book, JoRa} and others).
As was already known to Chaplygin, it is possible that  the reduced equations of motion  allow a Hamiltonian formulation, possibly after
a {\em time reparametrisation}.  In this scenario  one says that the system admits a \emph{Hamiltonization} \cite{EhlersKoiller,FedJov,BorMamHam,BalseiroGN}. 
The  theorem of Chaplygin on the reducing multiplier  \cite{ChapRedMult}  is one of the fundamental results in this area but its direct applicability is
limited to systems with a small number of degrees of freedom, possessing an invariant measure, and with a specific type of symmetry.

\subsection*{Contributions}
In this paper, we bring together ideas and results  from the theory of gauge momenta,    from singular nonholonomic reduction, and from Poisson geometry, to produce
new results in the area of Hamiltonization. Our research was inspired by the results of Borisov, Mamaev and Kilin \cite{BorMamChap,BorMam,BorMam2} who, by explicitly finding a Poisson bracket,
 showed that numerous systems---the Chaplygin
ball, a solid of revolution rolling without slipping on a plane or on a sphere,  or a uniform sphere rolling on a surface of revolution---admit a Hamiltonization.
 All of these examples share a fundamental property: {\em they possess $G$-invariant gauge momenta that  are Casimir functions of the 
Poisson bracket  on the reduced space}.

Our main result shows that
the above situation holds in considerable generality.
Namely, we prove that
a $G$-invariant nonholonomic system possessing $\ell$ independent $G$-invariant gauge momenta,  admits a description in almost Poisson form with respect to a $G$-invariant bracket
which upon reduction has  the gauge momenta as Casimirs.  This result gives  a positive answer to a question that was originally raised in \cite{LGNThesis} (see also Section 7 in \cite{LGN10}). 
Moreover, our proof is constructive and we give explicit formulae for the
bracket in terms of a choice of configuration coordinates $q^i$ and $G$-invariant momenta $\pi_\alpha$ (defined by an equivariant moving frame $\{X_\alpha\}$ for the
constraint distribution).

In our approach, the crucial object used to obtain the almost Poisson bracket  $\Pi_{nh}^\Lambda$ with the aforementioned properties, is a $G$-invariant 3-form $\Lambda$ defined on the configuration 
space. The construction of $\Lambda$ uses the kinetic energy metric and vector fields that generate the gauge momenta in a way that closely resembles the definition of the Cartan 3-form 
on a compact Lie group. We then consider the almost Poisson bracket $\Pi_{nh}$ introduced in \cite{vdSM} and use $\Lambda$ to perform a 
  {\em gauge transformation}   (as in \cite{BalseiroGN, SeveraWeinstein}).

 The need  to perform  a gauge transformation of $\Pi_{nh}$ to guarantee that the gauge momenta are Casimir functions
 of the reduced bracket is related to  a fundamental property of their generators: they are vector fields on configuration space which 
 are tangent to the group orbits but they are usually {\em not} infinitesimal generators of the action. Instead, they are so-called \emph{gauge symmetries} which are configuration-dependent linear combinations of the infinitesimal generators. Because of the variation of the coefficients, the corresponding gauge momenta in general fail to be Casimirs of the reduction of $\Pi_{nh}$. This scenario is not encountered in holonomic mechanics. See Section~\ref{sec:why gauge?}.

If the level sets of the gauge momenta on the reduced space are 2-dimensional, our result leads to a direct Hamiltonization of the problem in terms of a rank 2 Poisson structure 
(Corollary \ref{C:Rank2} in the text).  Two examples where this arises are the motion of a convex solid of revolution that rolls on a plane or on a sphere, and the motion of  a uniform sphere rolling on a convex surface of revolution. We stress that these
 rank 2 Poisson structures on the reduced space arise as the projection of the almost Poisson bracket $\Pi_{nh}^\Lambda$ defined on the unreduced space, 
and some dynamical properties may be deduced from this fact. 
In particular, certain components of the tensor $\Pi_{nh}^\Lambda$ are preserved by the flow of the unreduced equations  (Corollary \ref{C:Rank2}).

Another dynamical consequence may be drawn from our construction. If the characteristic distribution defined by the bracket induced by $\Pi_{nh}^\Lambda$ on the reduced space is integrable and the leaves coincide with the level sets of the resulting Casimirs, 
as is the case in all of the examples mentioned above, then the relative equilibria of the system are characterised as the critical points of the energy restricted to the level sets of the gauge momenta.

We mention  finally that it would be interesting to extend  the results of our paper to produce Casimir functions of the reduced bracket in the case where the gauge momenta are not $G$-invariant. Such Casimir functions would necessarily be nonlinear functions of the velocities of the system. This type of 
construction  could lead to a geometric understanding of the Hamiltonization of other examples like  
the multi-dimensional versions of the Chaplygin sphere and the Veselova system \cite{Jovan,FedJov},
or the nonholonomic hinge \cite{Biz}.  We will consider this question separately.

\subsection*{Previous work} The origin of the rank 2 Poisson  brackets given in  \cite{BorMam,BorMam2} for the reduced dynamics of the nonholonomic systems mentioned above
has  been considered by a number of authors using different approaches. 
In  \cite{FassoRank2} Fass\`o, Giacobbe and Sansonetto indicate that their existence is a consequence of
 the generic periodicity of the reduced dynamics. From their perspective, no insight about the dynamics of the unreduced system can be obtained from these structures. 
 This contrasts with our discussion above and with the content of Corollary \ref{C:Rank2}.

Ramos \cite{Ramos} studied these brackets from an algebraic perspective. He correctly indicates that there is no need to perform a rescaling to satisfy the Jacobi identity,
and notices that they allow an extension to the singular strata of the reduced space (that correspond to certain relative equilibria of the system). 
However, no link is made with any brackets
for the system on the unreduced space.

The treatment by Tsiganov \cite{Tsiganov} proceeds by doing the reduction in two steps and by
proposing an ansatz for a bracket at the intermediate stage. Then the author applies  `brute force calculations' hoping to  obtain brackets for the system with specific
properties. We mention that there is certain correspondence between his ansatz and with equation \eqref{E:PiBCoords} that gives the form of a bracket for the system
 that is obtained via a gauge transformation.

In \cite{Balseiro} Balseiro studies the Jacobiator of almost Poisson brackets that are obtained by gauge transformations of the nonholonomic bracket $\Pi_{nh}$ introduced in \cite{vdSM}.
An emphasis is given to the behaviour of this Jacobiator under reduction. A link between the constructions of \cite{Balseiro} and gauge momenta is suggested in \cite{BalseiroSanso}.

\subsection*{Outline}  The structure of the paper is the following. In Section \ref{S:Prelim} we give a quick review of the structure of the equations
of motion using quasi-velocities defined by moving frames. This allows us to introduce notation and  obtain working expressions for the almost Poisson bracket $\Pi_{nh}$ defined by van der Schaft and Maschke \cite{vdSM}, and known as the nonholonomic bracket.
In Section \ref{S:GaugeT} we show how, to a given a 3-form $\Lambda$ on the configuration space $Q$ one can associate an almost Poisson structure $\Pi_{nh}^\Lambda$ for the nonholonomic
dynamics that possesses the same properties as $\Pi_{nh}$. Our discussion follows the ideas of \cite{BalseiroGN, SeveraWeinstein} and gives 
explicit formulae for $\Pi_{nh}^\Lambda$. In Section 
\ref{S:GaugeReduction} we review some existing results about gauge momenta and reduction. 
For our purposes, the most convenient point of view for reduction is the one developed in  \cite{CDS-book} that applies in the case of non-free actions and follows a
Poisson perspective.

Our principal contributions are contained in Section \ref{S:Casimirs}. The material in Section~\ref{SS:AdaptedBases} establishes the preliminaries for
the formulation of  Lemma~\ref{L:Global3form} that states the existence of 
  the $G$-invariant 3-form $\Lambda$ that appears in the statement of Theorem~\ref{th:main}. 
 This theorem  states that the  gauge momenta of the system
 are Casimir functions of the reduction of 
$\Pi_{nh}^\Lambda$ and is our main result.
 Our proof is valid for free actions of $G$ on $Q$. For more general actions, we need to assume that the 3-form in 
 Lemma~\ref{L:Global3form}  admits a smooth extension to the points in $Q$ having  non-trivial isotropy (a situation that is encountered in 
 all examples that we consider). The proof of Lemma~\ref{L:Global3form} 
 is somewhat technical and is
presented in the Appendix. The special case in which the  difference between the rank of $D$ and the number of gauge momenta is less than 3,
is encountered in a great number of examples. In this case the
3-form in Lemma~\ref{L:Global3form} is unique and a local formula for the corresponding bracket  $\Pi_{nh}^\Lambda$
is presented in Section~\ref{SS:FinalObservations}. This section also contains a discussion of the need to modify the nonholonomic bracket $\Pi_{nh}$ by  $\Lambda$ and describes an open direction of research.


We treat two different examples in Sections \ref{S:ChSphere} and \ref{SS:SolidRev}. The first deals with the Chaplygin sphere and is meant to be an illustration of how the different
elements in our construction come together. In Section  \ref{SS:SolidRev} we consider  the rolling of a body of revolution on the plane. This example is more involved and shows
how our method works in the case of a non-free action. Up to an unnecessary conformal factor, we recover the bracket for the reduced system presented in \cite{BorMam}.

We would like to thank Francesco Fass\`o for comments on an early draft of this paper, and a referee for comments helping to improve the exposition.

Since completing this work, we have learned of recent work of Balseiro \cite{Balseiro2016} on the same example we treat in Section \ref{SS:SolidRev} where she shows that the bracket found in \cite{BorMam} (also without the conformal factor) can be found using gauge momentum methods similar to ours, although her method for modifying the nonholonomic bracket is chosen for that specific example. 
There is also ongoing work of the same author \cite{BalsYP} on the hamiltonization of a homogeneous ball rolling on a surface of revolution.

\section{Preliminaries}
\label{S:Prelim}

\subsection{The equations of motion --- Lagrangian approach}
\label{SS:Lagrangian}
Consider a nonholonomic system on an $n$-dimensional configuration manifold $Q$.  This consists of a Lagrangian $L:TQ\to \R$ which we assume to be of mechanical type, kinetic minus potential energy:
$$L(q,\dot q)=\tfrac12\left<\dot q,\,\dot q\right> - V(q),$$
where the kinetic term is given by a Riemannian metric $\left<\cdot ,\, \cdot \right>$ on $Q$, and $V(q)$ is the potential energy, and 
a regular non-integrable distribution $D\subset TQ$ of  rank $r<n$ that determines the nonholonomic constraints.   We assume that both the Lagrangian and the constraint distribution are time independent.  

In accordance with the Lagrange-D'Alembert principle, the equations of motion are given in bundle coordinates $(q^i,\dot q^i)$ by
\begin{equation}
\label{E:MotionBundleCoords}
\frac{d}{dt}\left ( \frac{\partial L}{\partial \dot q^i} \right )-  \frac{\partial L}{\partial  q^i} =R_i, \qquad i=1,\dots,n.
\end{equation}
Here $R_i:D\to \R$, denote  the components of the {\em constraint reaction force}. Such a reaction force is assumed to be {\em ideal}, namely\footnote{Unless the sum
over the indices is explicitly indicated, in this paper we use the Einstein summation convention.}
\begin{equation}
\label{E:Reaction}
R_i(q,\dot q)\dot q^i=0, \qquad \mbox{whenever} \qquad \dot q\in D_q,
\end{equation}
and is determined uniquely  by the condition that the nonholonomic constraints are satisfied. For convenience we consider 
$R$ as smoothly extended to $TQ$  (such an approach is taken  in \cite{Agostinelli}  for example). Then  
the system \eqref{E:MotionBundleCoords} is defined on all of  $TQ$ and its flow leaves $D$ invariant. The restriction of this system to $D$ determines
 a well defined vector field $Y_{nh}$ on $D$ whose integral curves satisfy  the equations 
of motion of our problem.

 As indicated by many different authors (e.g. \cite{EhlersKoiller,BMZ2010,Grab,BM2015} and references therein), 
the equations of motion are conveniently written by introducing  moving frames and the associated  {\em quasi-velocities}  adapted to the problem. 
This approach goes back to Hamel and allows one to write the equations  on $D$ explicitly, without any 
reference to the constraint force $R$, and will be very useful for our purposes. Following
the treatment in \cite{Grab}, let
$\{X_1, \dots, X_r\}$ be vector fields that form a local basis of sections of $D$ and let $\{X_{r+1}, \dots, X_{n}\}$ be a local basis of sections
of $D^\perp$ where the orthogonal complement is taken with respect to the kinetic energy metric. Any tangent vector $\dot q\in T_qQ$ can be written uniquely
as
\begin{equation*}
\dot q= v^\alpha X_\alpha(q) + v^AX_A(q),
\end{equation*}
for  coefficients $v^1, \dots, v^n$, that are the aforementioned quasi-velocities. Here and throughout, we use the following convention on the indices:
\begin{itemize}
\item greek indices  $\alpha, \beta, \gamma, \dots$ run from $1$ to $r$,
\item  latin indices $A, B, C, \dots$ run from $r+1$ to $n$,
\item latin indices $i,j,k, \dots$ run from $1$ to $n$.
\end{itemize}
We can now use $(q^i, v^\alpha, v^A)$ as coordinates
for $TQ$ and write $L=L(q^i, v^\alpha, v^A)$. Note that   $D\subset TQ$ is specified by the condition $v^A=0$ for all $A=r+1, \dots, n$,  so we can take $(q^i, v^\alpha)$ as coordinates for $D$. 

After a lengthy calculation using the chain rule, one can show that the equations \eqref{E:MotionBundleCoords} on $TQ$ are equivalent to 
\begin{equation}
\label{E:MotionTQQuasi}
\begin{split}
\frac{d}{dt} \left ( \frac{\partial L}{\partial v^\alpha} \right )  - \rho_\alpha^i  \frac{\partial L}{\partial q^i}&= -C_{\alpha j}^k v^j  \frac{\partial L}{\partial v^k}  ,  \quad \alpha=1, \dots, r, \\[8pt]
\frac{d}{dt} \left ( \frac{\partial L}{\partial u^A} \right )  - \rho_A^i  \frac{\partial L}{\partial q^i}&= -C_{A j}^k v^j  \frac{\partial L}{\partial v^k}  +\rho^i_AR_i, \quad A=r+1,\dots, n.
\end{split}
\end{equation}
In the above equation, the $q$-dependent coefficients $\rho_j^i$ and $C_{\alpha j}^k$ are defined by the relations
\begin{equation*}
\begin{split}
 X_j = \rho_j^i\partial_{q^i}, \qquad  \qquad [X_i, X_j ] = C_{ij}^k X_k, 
\end{split}
\end{equation*}
where $[\cdot , \cdot ]$ is the Lie bracket of vector fields.  In the derivation of \eqref{E:MotionTQQuasi} one has to make use of the relation
\begin{equation*}
\rho^i_\alpha R_i =0, \qquad \alpha = 1, \dots ,r,
\end{equation*}
which follows from \eqref{E:Reaction}.

The restriction of the system \eqref{E:MotionBundleCoords} on $TQ$ to $D$ can now be performed by substituting $v^A=0$ for all $A=r+1,\dots,n$ on the first set of equations
in  \eqref{E:MotionTQQuasi}. Note that our assumption that the Lagrangian is of mechanical type,
and that $X_\alpha$ and $X_A$ are orthogonal with respect to the kinetic energy metric, implies 
\begin{equation*}
\left .   \frac{\partial L}{\partial v^A} \right |_{v^B=0}=0.
\end{equation*}
If we write $L_c=L|_D$ for the constrained (or restricted) Lagrangian, the desired system on $D$ becomes
\begin{equation}
\label{E:MotionDQuasi}
\begin{split}
\frac{d}{dt} \left ( \frac{\partial L_c}{\partial v^\alpha} \right )  - \rho_\alpha^i  \frac{\partial L_c}{\partial q^i}&= -C_{\alpha \beta}^\gamma v^\beta  \frac{\partial L_c}{\partial v^\gamma}   \qquad \alpha=1, \dots, r.
\end{split}
\end{equation}
In coordinates the constrained Lagrangian is given by $L_c(q^i, v^\alpha)=L(q^i, v^\alpha,0)$
and satisfies
\begin{equation*}
L_c(q^i, v^\alpha) = \frac{1}{2} \mathcal{G}_{\beta \gamma}(q) v^\beta v^\gamma - V(q),
\end{equation*}
where  $\mathcal{G}_{\alpha \beta} := \langle  X_\alpha, X_\beta \rangle$.

Equations \eqref{E:MotionDQuasi} appear in \cite{Grab} and are complemented by the kinematic relations
\begin{equation}
\label{E:MotionKinematic}
\dot q^i= \rho_\alpha^iv^\alpha, \qquad i=1,\dots, n,
\end{equation}
that follow from the definition of the quasi-velocities. Note that the second set of equations in  \eqref{E:MotionTQQuasi} does not give any information about the
dynamics in $D$. Instead it serves to determine the reaction force $R$.

It is useful to note that,  since the vector fields $X_\alpha$ and $X_A$ are orthogonal, the coefficients $C_{\alpha \beta}^\gamma$ satisfy
$\mathcal{G}_{\gamma \delta} C_{\alpha \beta}^\delta=  \langle [X_\alpha, X_\beta ] , X_\gamma \rangle $.
The coefficients
\begin{equation}
\label{E:defCalbega}
C_{\alpha \beta \gamma} := \langle [X_\alpha, X_\beta ] , X_\gamma \rangle = \mathcal{G}_{\gamma \delta} C_{\alpha \beta}^\delta
\end{equation}
will play an important role in our construction. 

\subsection{The equations of motion --- Hamiltonian approach}
\label{SS:MotionHamiltonian}  
The Hamiltonian approach is defined on the dual bundle $D^*$.  For our treatment   it is convenient to realise this abstract bundle as a vector subbundle 
of  $T^*Q$, but, as is well known from
basic linear algebra, there is no canonical way to do this. However, 
if   $W$ is  {\em any} fixed subbundle of $TQ$ complementary to $D$, that is, $TQ=D\oplus W$, then there is a natural identification of $D^*$ 
 with the annihilator $W^\circ$ which is a subbundle of  $T^*Q$.  One possible choice for $W$ is of course $D^\perp$, but in examples other choices may simplify coordinate calculations. 
  In what follows we assume that a choice of $W$ has been made. The 
sections of $D^*$ can therefore be interpreted as 1-forms on $Q$ that annihilate $W$, and we can operate with them as one normally does.

Let $\{\mu^\alpha\}$ be  a local basis of sections of $D^*$ that is dual to  $\{X_\alpha\}$. Namely, $\mu^\alpha$ are  locally defined 1-forms
on $Q$ that vanish along $W$ and satisfy $\mu^\alpha(X_\beta)=\delta_\beta^\alpha$ (Kronecker delta). Any element of $D^*$ on the fibre $D_q^*$ over $q\in Q$ can be
written uniquely as a linear combination $\pi_\alpha \mu^\alpha(q)$ for certain scalars $\pi_1, \dots, \pi_r$. In this way 
one can use $(q^i,\pi_\alpha)$ as coordinates on $D^*$.

The restriction of the kinetic energy to $D$ is non-degenerate and allows one to define a constrained Legendre transform $\mbox{Leg}_c$ that is
a vector bundle isomorphism between $D$ and $D^*$. It is the restriction to $D$ of the standard Lagrange transform in mechanics. Explicitly we have
\begin{equation*}
\mbox{Leg}_c :D\to D^*, \qquad (q^i, v^\alpha) \mapsto \left (q^i, \pi_\alpha = \frac{\partial L_c}{\partial v^\alpha}= \mathcal{G}_{\alpha \beta}v^\beta \right  ).
\end{equation*}
The constrained Hamiltonian $H_c:D^*\to \R$ is the energy of the system and is defined by
\begin{equation*}
H_c(q^i,\pi_\alpha)= \pi_\beta v^\beta - L_c(q^i, v^\beta),
\end{equation*}
where it is understood that $v^\beta$ is written in terms of $\pi_\alpha$ via the inverse Legendre transform $\mbox{Leg}_c^{-1}$. Explicitly we have
\begin{equation}
\label{E:HamExp}
H_c(q^i,\pi_\alpha) = \tfrac{1}{2}\mathcal{G}^{\alpha \beta} \pi_\alpha \pi_\beta + V(q),
\end{equation}
where $\mathcal{G}^{\alpha \beta}$ denotes the inverse matrix of  $\mathcal{G}_{\alpha \beta}$; namely, $\mathcal{G}^{\alpha \gamma}\mathcal{G}_{\gamma \beta}=\delta_\beta^\alpha$.
It is a straightforward calculation to show that
\begin{equation*}
v^\alpha = \frac{\partial H_c}{\partial \pi_\alpha}, \qquad \frac{\partial H_c}{\partial q^i} = -  \frac{\partial L_c}{\partial q^i},
\end{equation*}
and hence, the equations \eqref{E:MotionDQuasi}, \eqref{E:MotionKinematic}, are equivalent to the following first order system on $D^*$:
\begin{equation}
\label{E:MotionHam}
\dot q^i= \rho_\alpha^i \frac{\partial H_c}{\partial \pi_\alpha}, \qquad \dot \pi_\alpha = -  \rho_\alpha^i \frac{\partial H_c}{\partial q^i} - 
C_{\alpha \beta}^\gamma \pi_\gamma  \frac{\partial H_c}{\partial \pi_\beta}.
\end{equation}
We denote the corresponding vector field by $X_{nh}$, which is nothing other than the push-forward of $Y_{nh}$ by $\mbox{Leg}_c$.  It is readily seen that the above equations can be written as
\begin{equation*}
\dot q^i= \{q^i,H_c\}_{nh}, \qquad \dot \pi_\alpha= \{\pi_\alpha,H_c\}_{nh}
\end{equation*}
where $\{\cdot , \cdot \}_{nh}$ denotes the   bracket of functions on $D^*$ that is skew-symmetric,  satisfies the Leibniz rule, and  is locally defined by the relations
\begin{equation*}
\{q^i, q^j\}_{nh}=0, \qquad \{q^i, \pi_\alpha \}_{nh}=\rho_\alpha^i, \qquad  \{\pi_\alpha , \pi_\beta \}_{nh}=-C_{\alpha \beta}^\gamma \pi_\gamma.
\end{equation*}
This bracket coincides with the one that was first introduced by van der Schaft and Maschke in \cite{vdSM} where it is shown that the 
 Jacobi identity is satisfied if and only if the distribution $D$ is integrable. For a non-integrable $D$ one speaks of
an {\em almost Poisson bracket} or a {\em pseudo-Poisson bracket}.  Several intrinsic constructions and geometric interpretations of the  bracket  $\{\cdot , \cdot \}_{nh}$ are 
available in the literature, see e.g. \cite{Marle,Cantrijn99,Ibort,Leon,CDS-book} and the references therein. The presentation given above  follows roughly the approach and notation of \cite{Leon}.

Denote by $\Pi_{nh}$ the bivector on $D^*$  defined by the bracket $\{\cdot , \cdot \}_{nh}$. Its expression in coordinates  is
\begin{equation}
\label{E:PiCoords}
\Pi_{nh}=\rho^i_\alpha \partial_{q^i}\wedge \partial_{\pi_\alpha} -\frac 12 C_{\alpha \beta}^\gamma\pi_\gamma \partial_{\pi_\alpha}\wedge \partial_{\pi_\beta}.
\end{equation}
The Hamiltonian\footnote{strictly speaking, $X_f$ is not a Hamiltonian vector field, since $\Pi_{nh}$ is not a Poisson bracket; it is only an `almost' Hamiltonian vector field. We will however call these Hamiltonian vector fields throughout.} vector field $X_f$ associated to 
$f\in C^\infty(D^*)$ is defined by $X_f:=\Pi_{nh}^\sharp(df)$. Then clearly $X_{nh}=X_{H_c}$.  
Recall the standard notation: if $\Pi$ is a bivector field on $M$ then $\Pi^\sharp:T^*M\to TM$ is the associated bundle map; similarly if $\Xi$ is a 2-form on $M$ then $\Xi^\flat$ is the bundle map $\Xi^\flat:TM\to T^*M$.

\subsection{Second order vector fields} The vector field   $Y_{nh}$ on $D$ defined by the 
Lagrangian equations of motion \eqref{E:MotionDQuasi} and  \eqref{E:MotionKinematic}  is {\em second order} with respect to the vector bundle structure of 
 $D$  over the configuration manifold $Q$. That is, $Y_{nh}$ satisfies
 \begin{equation*}
T\tau_D \circ Y_{nh} = \mbox{id}_D,
\end{equation*}
where $\tau_D:D\to Q$ is the bundle projection.
This condition naturally extends the standard definition of second order vector fields on $TQ$, see e.g.\ \cite{MaRa}.  In local coordinates
we have $T\tau_D (\dot q, \dot v)=\dot q$, so the condition
of being second order is written as 
\begin{equation*}
\dot q(t)= v^\alpha(t)X_\alpha(q(t))\in D_{q(t)},
\end{equation*}
 for any integral curve $(q(t),v^\alpha(t))$ of $Y_{nh}$.

 The identification of $D$ with $D^*$ via the constrained Legendre transform suggests the following.
\begin{definition}
A vector field on $Z$ on $D^*$ is \emph{second order} if its pull-back to $D$ by $\mbox{Leg}_c$ is second order.
\end{definition}
In coordinates this means that a second order vector field $Z$ on $D^*$ satisfies 
$$T\tau_{D^*} (Z(q,\pi)) =\mathcal{G}^{\alpha\beta}\pi_\beta X_\alpha(q)$$
where $\tau_{D^*}:D^*\to Q$ is the projection. The vector field $X_{nh}$ on $D^*$ defined by the equations \eqref{E:MotionHam} is the push forward of $Y_{nh}$ by the constrained Legendre
transform and it is clearly second order.

Let $\beta^A$ be linearly independent 1-forms on $Q$ that span the annihilator $D^\circ$ and let $\mathcal{E}$ be the distribution on $D^*$ defined by
the joint annihilator of $\{ \tau^*_{D^*}\beta^A \}$ on $D^*$. It is clear that $\mathcal{E}$  is  well defined (independent of the choice of basis of $D^\circ$) and is a regular distribution of rank $2r$ on $D^*$. (A {\em regular} distribution is one of constant rank.)

The importance of the distribution  $\mathcal{E}$ in the geometric formulation of nonholonomic mechanics seems to have first been noticed in
\cite{Weber} where it is shown that its fibres are symplectic (with respect to the canonical symplectic form on the ambient space $T^*Q$).
$\mathcal{E}$ is  a crucial ingredient in the formulations in e.g. \cite{BS93}, \cite{CDS-book}, \cite{Ibort}, \cite{JoRa}. In coordinates we
have
\begin{equation*}
\mathcal{E}_{(q,\pi)}= \mbox{span} \left \{ \rho^i_\alpha \partial_{q^i} ,  \partial_{\pi_\alpha} \right \}.
\end{equation*}
We collect some  properties of  $\mathcal{E}$ in the following proposition whose proof can be given using the above expression.

  \begin{proposition} 
  \label{P:PropertiesPi}
The distribution $\mathcal{E}$ has the following properties:
  \begin{enumerate} 
\item it is integrable if and only if $D$ is integrable;
\item a vector field $Z$ on $D^*$ is second order if 
and only if it is a section of $\mathcal{E}$;
\item it coincides with the characteristic distribution of\/ $\Pi_{nh}$: that is,  $\Pi_{nh}^\sharp(T^*D^*)=\mathcal{E}$.
\end{enumerate}
\end{proposition}
As a consequence, all of the
Hamiltonian vector fields on $D^*$ generated by $\Pi_{nh}$ are second order.

 \section{Gauge transformations associated to 3-forms}
 \label{S:GaugeT}
 
 A particular procedure to construct a family of almost Poisson structures that describe the dynamics of a nonholonomic system 
 was presented in \cite{LGN10}. Later, this method was put into a solid  geometric context in \cite{BalseiroGN} by relating it to the  {\em gauge transformations} of Poisson brackets by 2-forms as introduced by \v{S}evera and Weinstein \cite{SeveraWeinstein}.
Here we modify this construction by starting from a 3-form $\Lambda$ on $Q$.  In Section \ref{S:Casimirs} we apply this construction using specific 3-forms defined by the nonholonomic geometry in the presence of symmetry.

Let $\Lambda$ be a section of $\wedge^3(D^*)$. Using our identification of 
 $D^*$ with the annihilator $W^\circ\subset T^*Q$ of a complement $W$ of $D$ on $TQ$, we can interpret 
  $\Lambda$ as a 3-form on $Q$ that vanishes upon contraction with tangent vectors to  $W$.
  Fix a basis of local basis of sections $\{X_\alpha\}$ of $D$. The  local expression for the 3-form $\Lambda$ is
  \begin{equation*}
\Lambda=\frac{1}{6}B_{\alpha \beta \gamma} \, \mu^\alpha \wedge \mu^\beta\wedge \mu^\gamma,
\end{equation*}
where, as before, $\{\mu^\beta\}$ are locally defined 1-forms on $Q$  annihilating $W$, such that $\mu^\beta(X_\alpha)=\delta_\alpha^\beta$. Here, the 
    coefficients  $B_{\alpha \beta \gamma}$  are alternating (that is, skew-symmetric with respect to transpositions of $\alpha, \beta, \gamma$) and are given by
\begin{equation}
\label{eq:Bcoeff-general}
B_{\alpha \beta \gamma}:= \Lambda(X_\alpha,X_\beta,X_\gamma).
\end{equation}

 We define the 2-form $\Xi$ on $D^*$ as the contraction
of the pull-back $\tau^*_{D^*}\Lambda$ of $\Lambda$ to $D^*$ with any second order vector field $Z$ on $D^*$.  The 2-form $\Xi$ is independent of the choice of $Z$, it is semi-basic, and has local expression
\begin{equation*}
\Xi=\frac{1}{2}B_{\alpha \beta}^\gamma \pi_\gamma \, \mu^\alpha \wedge \mu^\beta,
\end{equation*}
 where $B_{\alpha \beta}^\gamma:=\mathcal{G}^{\gamma \delta} B_{\alpha \beta \delta}$.

 \begin{lemma} \label{Lem:endomorphism}
Given any 3-form $\Lambda$ as above and writing $\Xi$ for the resulting 2-form, the map
 \begin{equation*}
\mathrm{Id}_{TD^*}+ \Pi_{nh}^\sharp \circ \Xi^\flat
\end{equation*}
is an invertible endomorphism of $TD^*$.
 \end{lemma}
 
 \begin{proof}
 We work in local coordinates  with  the notation introduced in Section \ref{S:Prelim}. We have
 \begin{equation*}
\mu^\alpha=\bar \rho^\alpha_i \, dq^i, \qquad \alpha=1, \dots, r,
\end{equation*}
where the duality between  $\{\mu^\alpha\}$ and $\{X_\beta\}$ implies $\bar \rho^\alpha_i \rho_\beta^i=\delta^\alpha_\beta$.
We shall denote by $\rho$ the $n\times r$ matrix with entries $ \rho^\beta_i$, and by $\bar \rho$ the $r\times n$  matrix with entries $\bar \rho^\alpha_i$. The duality condition becomes  $\bar \rho \rho=\mathrm{Id}_{r\times r}$.

  The matrix representations of
$\Xi^\flat$ and $ \Pi_{nh}^\sharp$ with respect to the respective bases $\{\partial_{q^i}, \partial_{\pi_\alpha}\}$ and $\{dq^i,d\pi_\alpha\}$ of $T_{(q,\pi)}D^*$ and $T_{(q,\pi)}^*D^*$
are given in block form by
\begin{equation}
\label{E:MatrixPiB} 
\Xi^\flat =  \begin{pmatrix} \bar \rho^T \mathcal{B} \bar \rho &0 \\ 0 & 0 \end{pmatrix}, \qquad  \Pi_{nh}^\sharp=\begin{pmatrix} 0 & \rho \\ -\rho^T & -\mathcal{C} \end{pmatrix},
\end{equation}
where $\mathcal{C}$ and $\mathcal{B}$ are $r\times r$  skew-symmetric matrices with entries
\begin{equation*}
\mathcal{C}_{\alpha \beta} = C_{\alpha \beta}^\gamma \pi_\gamma, \qquad \mathcal{B}_{\alpha \beta} = B_{\alpha \beta}^\gamma \pi_\gamma.
\end{equation*}
Performing the matrix algebra, one finds  
\begin{equation}
\label{E:MatrixAux}
 \mbox{Id}_{TD^*}+ \Pi_{nh}^\sharp \circ \Xi^\flat = \begin{pmatrix} I &0 \\ -\mathcal{B}\bar \rho & I \end{pmatrix}
\end{equation}
which is clearly invertible.
 \end{proof}
 
 Following \cite{BalseiroGN} and \cite{SeveraWeinstein} we define the bivector $\Pi_{nh}^\Lambda$ by 
 \begin{equation}
 \label{E:DefPiB}
(\Pi_{nh}^\Lambda)^\sharp:=(\mbox{Id}_{TD^*}+ \Pi_{nh}^\sharp \circ \Xi^\flat)^{-1} \circ \Pi_{nh}^\sharp.
\end{equation}
This exists by virtue of the lemma above. We say that the bracket defined by $\Pi_{nh}^\Lambda$ is obtained by a \emph{gauge transformation} of 
the nonholonomic bracket  $\{\cdot , \cdot \}_{nh}$ by the 3-form $\Lambda$. (Note that in \cite{BalseiroGN} these would be denoted $\Pi_{nh}^\Xi$, but in the present context the fundamental object is the 3-form $\Lambda$ rather than its contraction $\Xi$.)

\begin{theorem}
\label{T:PropPiB}
The bivector field\/ $\Pi_{nh}^\Lambda$ has the following properties:
\begin{enumerate}
\item its characteristic distribution is $\mathcal{E}$, that is $(\Pi_{nh}^\Lambda)^\sharp (TD^*)=\mathcal{E}$,
\item it describes the nonholonomic dynamics; namely
\begin{equation}
\label{E:KeyPropB}
(\Pi_{nh}^\Lambda)^\sharp(dH_c)=X_{nh},
\end{equation}
\item it  is given in local coordinates by
\begin{equation}
\label{E:PiBCoords}
\Pi_{nh}^\Lambda=\rho^i_\alpha \partial_{q^i}\wedge \partial_{\pi_\alpha} +\frac 12 \left (B_{\alpha \beta}^\gamma - C_{\alpha \beta}^\gamma \right )\pi_\gamma \, \partial_{\pi_\alpha}\wedge \partial_{\pi_\beta}.
\end{equation}
\end{enumerate}
\end{theorem}

 \begin{proof}
(i): this follows directly from Proposition \ref{P:PropertiesPi} and equation \eqref{E:DefPiB} that defines $\Pi_{nh}^\Lambda$. \\
 (ii): since the 2-form $\Xi$ was defined as the contraction of the 3-form $\tau^*_{D^*}\Lambda$ with any second order vector field $Z$ on $D^*$,  and $X_{nh}$ is a  second order
 vector field on $D^*$, we can write 
 \begin{equation*}
\Xi(\cdot , \cdot )=\tau^*_{D^*}\Lambda(X_{nh}, \cdot , \cdot),
\end{equation*}
and therefore
\begin{equation*}
\Xi^\flat(X_{nh})=0.
\end{equation*}
It follows that $(\mbox{Id}_{TD^*}+ \Pi_{nh}^\sharp \circ \Xi^\flat)(X_{nh})=X_{nh}$, and hence \eqref{E:KeyPropB} holds since $\Pi_{nh}^\sharp(dH_c)=X_{nh}$. \\
(iii):  Finally, in the notation of the proof of Lemma \ref{Lem:endomorphism}, and in view of \eqref{E:MatrixAux} we have
 \begin{equation*}
 (\mbox{Id}_{TD^*}+ \Pi_{nh}^\sharp \circ \Xi^\flat)^{-1} = \begin{pmatrix} I &0 \\ \mathcal{B}\bar \rho & I \end{pmatrix}
\end{equation*}
 which, combined with \eqref{E:MatrixPiB}, gives
 \begin{equation*}
(\Pi_{nh}^\Lambda)^\sharp= \begin{pmatrix} I &0 \\ \mathcal{B}\bar \rho & I \end{pmatrix} \begin{pmatrix} 0 & \rho \\ -\rho^T & -\mathcal{C} \end{pmatrix} = 
 \begin{pmatrix} 0 & \rho \\ -\rho^T &\mathcal{B} -\mathcal{C} \end{pmatrix} ,
\end{equation*}
and this is equivalent to \eqref{E:PiBCoords}.
 \end{proof}
 
 The above theorem shows that the equations of motion \eqref{E:MotionHam} can be formulated in Hamiltonian form
 with respect to the bivector $\Pi_{nh}^\Lambda$, for any 3-form $\Lambda$.  The Jacobi identity fails for $\Pi_{nh}^\Lambda$ since the characteristic distribution $\mathcal{E}$ is non-integrable.  Note that all of the  Hamiltonian vector fields associated to $\Pi_{nh}^\Lambda$ are tangent to $\mathcal{E}$ and therefore are second order
 vector fields on $D^*$.

 \begin{remarkth} At many points in this section, we  liberally refer to $\Lambda$ as a 3-form, and we will  continue to do so 
 throughout the paper. Strictly speaking $\Lambda$ is  a section of $\wedge^3(D^*)$ and it is only 
  in virtue of the identification $D^*=W^\circ$, that we may interpret it as a 3-form on $Q$ (that annihilates $W$). This issue is taken into
  consideration in the construction in  Lemma~\ref{L:Global3form} of the specific $\Lambda$ that is used to prove our main result.
  \end{remarkth}

 \section{Gauge momenta and symmetry reduction}
 \label{S:GaugeReduction}
 
 The existence of first integrals that are linear in velocity, or momentum variables, for nonholonomic systems
has received a great deal of attention (see e.g. \cite{FassoJGM} and the references therein). In this section we outline some of the known results in the field and prove Theorem 
\ref{Th:isometries on D} which will be used in our construction in Section \ref{S:Casimirs} below. Similar versions of this theorem are available in the literature (see the
discussion in \cite{FassoJGM}).

\subsection{Linear first integrals of nonholonomic systems.} 

Consider first the Hamiltonian formulation in terms of the equations of motion \eqref{E:MotionHam}  on $D^*$. 
Linear functions on $D^*$ are sections of $D^{**}=D$ so we can
naturally make a one to one correspondence between vector fields that take values on $D$ and linear functions on the phase space $D^*$.
This observation goes back to Iliev \cite{Iliev1, Iliev2} and to some extent Agostinelli \cite{Agostinelli}. Let $Z$ be a vector field on $Q$ taking  values on $D$. It can  be written as a linear combination of the basis of sections of $D$ as
\begin{equation*}
Z=Z^\alpha(q)X_\alpha(q),
\end{equation*}
for certain functions $Z^\alpha\in C^\infty(Q)$. We denote by $p_Z\in C^\infty(D^*)$ the linear function associated to $Z$. In terms of the 
coordinates $(q^i,\pi_\alpha)$ we have
\begin{equation*}
p_Z(q,\pi)=Z^\alpha(q)\pi_\alpha.
\end{equation*}
In particular note that $p_{X_\alpha}=\pi_\alpha$ for all $\alpha=1, \dots, k$.

Now consider the Lagrangian formulation given by equations \eqref{E:MotionDQuasi}, \eqref{E:MotionKinematic} defined on $D$.
This approach is followed by many recent references \cite{Fasso1,Fasso2}.  With a slight abuse of notation 
we denote the  function $p_Z\circ \mathrm{Leg}_c  \in C^\infty(D)$  also by $p_Z$. It has the local expression
\begin{equation*}
p_Z(q,v)=Z^\alpha(q)\frac{\partial L_c}{\partial v^\alpha}=Z^\alpha(q)\mathcal{G}_{\alpha \beta}(q)v^\beta.
\end{equation*}

\begin{definition}
The vector field $Z$ taking values on $D$ is called   the \emph{generator} of the linear function $p_Z$. 
\end{definition}
We stress that the generator of a linear function is uniquely determined by the condition that it
is a section of $D$. Also, linear functions on $D$ (or $D^*$) are independent over the set of points where the generating vector fields are
linearly independent.

As it is shown in \cite{Iliev2}, the evolution of $p_Z$ along the nonholonomic Lagrangian system  \eqref{E:MotionDQuasi}, \eqref{E:MotionKinematic} on
$D$ is given by
\begin{equation}
\label{E:pZdot}
\dot p_Z = \left . Z^{TQ}[L] \right |_D
\end{equation}
where $Z^{TQ}$ denotes the {\em tangent lift} of the generator vector field $Z$ (the expression for $Z^{TQ}$ in bundle coordinates $(q^i, \dot q^i)$ is given 
 in \eqref{E:TanLift} below).

We now present an alternative characterisation of the condition that $p_Z$ is a first integral. Recall that a vector field $Z$ on $Q$ acts by infinitesimal isometries if the Lie derivative $\pounds_Z\mathcal{G}$ of the metric along $Z$ vanishes. This suggests the following definition.

\begin{definition}
A vector field $Z$ on $Q$ acts by \emph{infinitesimal isometries on} $D$ if $(\pounds_Z\mathcal{G})(u,u)=0$ for all $u\in D$.
\end{definition}

Note that this does not assume that the flow defined by $Z$ preserves the distribution $D$. The following is a result from \cite{Fasso1}, where they allow $Z$ to be a section of the {\em reaction annihilator distribution}, which contains $D$.

\begin{theorem}
\label{Th:isometries on D}
Consider a nonholonomic Lagrangian system with constraint distribution $D$ and Lagrangian $L(q,u)=\mathcal{G}(u,u)-V(q)$.  Let $Z$ be a section of the distribution $D$. 
Then the momentum $p_Z$ generated by $Z$ is a first integral of the Lagrangian system if and only if $Z$ acts by infinitesimal isometries on $D$ and annihilates the potential energy $V$.
\end{theorem}

\begin{proof}
We have $\dot p_Z=\left.Z^{TQ}[L]\right |_D$, so $p_Z$ is conserved if and only if the right hand side vanishes.  We show this is equivalent to annihilating the potential and kinetic energies separately, and then show the latter is equivalent to being an infinitesimal isometry on $D$ (which follows from the lemma below).  One implication is clear \emph{a fortiori}. For the converse, if $Z^{TQ}$ annihilates $L$ on $D$, then restricting to the zero section shows that $Z[V]=0$. Since $V$ is a function on $Q$, it follows that $Z^{TQ}$ annihilates its pull-back to $TQ$ (and hence to $D$).  Since $L$ and $V$ are both annihilated by $Z^{TQ}$ (on $D$), it follows that $Z^{TQ}$ also annihilates the metric (kinetic energy) restricted to $D$.  That this is equivalent to $Z$ being an isometry on $D$ follows from the following lemma.
\end{proof}

\begin{lemma}
Let $Z$ be any vector field on $Q$,  and $\mathcal{G}$ the metric tensor. Then as functions on  $TQ$,
$$(\pounds_Z\mathcal{G})(u,u) = {Z^{TQ}}\left [ \mathcal{G}(u,u) \right ].$$
\end{lemma}

\begin{proof}
In local bundle coordinates, the tangent lift of $Z$ is
\begin{equation}
\label{E:TanLift}
Z^{TQ} = Z^j\partial_{q^j} + \frac{\partial Z^k}{\partial q^\ell}\dot q^\ell\partial_{\dot q^k}.
\end{equation}
Applying this to $\mathcal{G}(u,u)=\mathcal{G}_{ij}u^iu^j$ gives
$$Z^{TQ}[\mathcal{G}(u,u)] = Z^k\frac{\partial \mathcal{G}_{ij}}{\partial q^k} u^iu^j + 2 \mathcal{G}_{kj}\frac{\partial Z^k}{\partial q^\ell}u^\ell u^j.$$
But this is exactly the expression for $(\pounds_Z\mathcal{G})(u,u)$ --- see for example \cite[p.55]{Jost}. 
\end{proof}

\subsection{Gauge momenta}

An important class of linear first integrals that may exist  in the presence of a symmetry group is that of {\em gauge momenta}. 
This terminology was first used in \cite{Bates} where the authors indicate the existence of this kind of integrals in some classical
examples of nonholonomic systems. Further research on their properties can be found in \cite{Fasso2}, \cite{Fasso3}, where they are called  {\em horizontal} gauge momenta.

 Let $G$  be a Lie group that acts properly on $Q$, and suppose that
the lift of $G$ to $TQ$ leaves the Lagrangian $L$ and the constraint distribution $D$ invariant. It follows that $G$ acts by isometries on $Q$ (with respect to the
kinetic energy metric), and that the potential $V\in C^\infty(Q)$ is $G$-invariant. Define by $\mathcal{S}$ the, possibly non-regular, distribution on $Q$
defined by
\begin{equation*}
\mathcal{S}_q:=D_q \cap \g\cdot q.
\end{equation*}
Here $\g$ is the group's Lie algebra and $\g\cdot q$ denotes the tangent space to the group orbit at $q$.  We assume that this distribution is regular on $Q_f$, and only changes rank possibly at points where the action of $G$ fails to be free.

\begin{definition}
A linear first integral of a nonholonomic system is called a \emph{gauge momentum} associated to the $G$-action if its unique generator
on $D$ is a section of $\mathcal{S}$.
\end{definition}

We will be especially interested in  $G$-\emph{invariant gauge momenta}. In addition to being a section of $\mathcal{S}$, the generator of such integrals is $G$-equivariant.

\subsection{Reduction} 
We present a basic outline of the almost Poisson reduction of nonholonomic systems
with possibly non-free actions.
We continue to work under the assumptions  that were introduced above. Namely,
there is an action of the Lie group $G$ on the configuration space $Q$ 
whose lift to $TQ$ preserves the
Lagrangian $L$ and the constraint distribution $D$. 

%

Via the constrained Legendre transform, the restricted action on $D$ defines a $G$-action on $D^*$ that leaves the constrained Hamiltonian $H_c$ \eqref{E:HamExp} and the bivector $\Pi_{nh}$  \eqref{E:PiCoords}
invariant.  The {\em reduced Hamiltonian} is the unique function, which we also denote $H_c$, on the orbit space $D^*/G$ whose pull-back to $D^*$ is $H_c$.  If the $G$-action on $Q$ is free then this orbit space can be expressed as a bundle of rank $r$ over the `shape space' $Q/G$ (where $r$ is the rank of $D$).  In general, if the  action is not free, the quotient $D^*/G$ is a stratified space.

As is usual, one identifies smooth functions on $D^*/G$ with smooth $G$-invariant functions on $D^*$.  Since the nonholonomic (almost) Poisson structure is invariant,  it follows that $\{f,g\}_{nh}$ is invariant whenever $f$ and $g$ are, and hence $\{ \cdot , \cdot \}_{nh}$  descends to an almost Poisson structure on $D^*/G$, which we continue to denote $\{ \cdot , \cdot \}_{nh}$.  One can, by the usual formula, define  Hamiltonian `vector fields' on $D^*/G$ from this bracket, and it turns out that these are genuine vector fields on, or tangent to,  each stratum.  In particular, the Hamiltonian vector field defined by the reduced Hamiltonian $H_c$ is the \emph{reduced Hamiltonian vector field}. The flow of this reduced Hamiltonian vector field is the projection to $D^*/G$ of the flow of the original Hamiltonian vector field on $D^*$ given by equations \eqref{E:MotionHam}. 

Further details about the reduction procedure when the action is not free can be found in the book of Cushman, Duistermaat and \'Sniatycki \cite{CDS-book}.

The main point of our paper is that given any $G$-invariant 3-form $\Lambda$ on $Q$ as constructed in Section \ref{S:GaugeT}, the reduction outlined above follows 
{\em mutatis mutandis} for the bracket on $D^*$ defined by the bivector $\Pi_{nh}^\Lambda$. In the next section
we see how to choose $\Lambda$ so that the invariant gauge momenta are Casimir functions of the reduced system.

\section{Invariant gauge momenta are Casimirs of a reduced bracket}
\label{S:Casimirs}

In this section we continue to work under the assumption that there is a Lie group $G$ acting on $Q$ and 
preserving both the mechanical Lagrangian $L$ and the constraint distribution $D$.
Moreover, we shall assume that the $G$-action on $Q$ is proper and that the
isotropy of a generic point $q\in Q$ is trivial. We denote by $Q_f$ the set of points in $Q$ having trivial isotropy. Then $Q_f$ is
a $G$-invariant  open dense subset of $Q$ and the restriction of the $G$ action to $Q_f$ is free and proper.

We suppose that, on $Q_f$ there exist $\ell$ linearly independent gauge momenta with $G$-equivariant generators
$Z_b$, for $b=1,\dots, \ell$. We note that  linear combinations of these vector fields with {\em constant} coefficients are 
also generators of gauge momenta.

\subsection{Adapted bases and preliminary results} \label{SS:AdaptedBases}

\begin{definition}
\label{E:DefAdaptedBasis}
A basis $\{X_\alpha\}$ of sections of $D$ defined on an open subset $U$ of $Q_f$ is said to be \emph{symmetry-momentum adapted} (or simply \emph{adapted}) if it
can be written as $\{X_\alpha\}= \{Z_b, Y_I\}$ where:
\begin{enumerate}
\item the vector fields $\{Z_b\}$ generate gauge momenta,
\item the vector fields $Z_b$ and $Y_I$ are $G$-equivariant.
\end{enumerate}
\end{definition}
Here we refine the convention on the indices introduced in Section \ref{SS:Lagrangian} where $\alpha,\beta,\dots$ run from $1,\dots,r$ (where $r=\text{rank}(D)$) by using
\begin{itemize}
\item lower case latin indices\footnote{we do not use the letter $a$ since its 
 typography  is very similar to that of $\alpha$.} $b, c, d, \dots$ running from $1$ to $\ell$, 
 \item  upper case latin indices $I,J,K, \dots$ running from $\ell +1$ to $r$.
\end{itemize}

The following lemma shows that it is possible to extend locally the gauge momentum generators $\{Z_b\}$ to a minimal set of generators of $D$ that is adapted at points of $Q_f$.  Note that if a distribution is regular, then a minimal set of generators is a set of vector fields that defines a basis of the distribution at each point.

\begin{lemma}
\label{L:ExistBasis}
 Let $Z_b$, for $b=1,\dots,\ell$, be given linearly independent equivariant vector fields in $D$ and let $q\in Q_f$.  There is a $G$-invariant neighbourhood $U$ of $q$ in $Q_f$ on which there exist equivariant vector fields $Y_I$ ($I=\ell+1,\dots,r$) such that $\{Z_b,Y_I\}$ generate sections of $D$ on $U$.
\end{lemma}

\begin{proof}
Let $S$ be a submanifold of $Q_f$ for which $T_qS\oplus T_q(G\cdot q) = T_qQ$ (that is, $S$ is a slice to the orbit). Then the natural map $G\times S\to Q$ defined by $(g,s) \longmapsto g\cdot s$ defines an equivariant diffeomorphism in a neighbourhood of $Q$ and hence in a neighbourhood of $G\cdot q$ (by the inverse function theorem). The image $U$ of such a neighbourhood is called a \emph{tubular neighbourhood} of $q$ (or of the orbit $G\cdot q$).  

Let $V$ be any vector field on $Q$ defined in a neighbourhood of $q$. Then its restriction to $S$ can be extended to an equivariant vector field $Y$ on the tubular neighbourhood simply by the formula
$$Y(g,s) = TgV(s).$$
If $V$ is a section of a $G$-invariant distribution, then so is the resulting equivariant vector field $Y$.

Since $D$ is a smooth distribution, the given set of vector fields $\{Z_b\}$ can be extended to a minimal set of generators $\{Z_b,V_I\}$ of $D$.  For each of the $V_I$, restrict to $S$ and extend by equivariance to define $Y_I$ as above.  These span $D$ by dimension count: they remain linearly independent in a neighbourhood of $q$ (as they have the same value at $q$ as the $V_I$), and they are sections of $D$.
\end{proof}

The proof of the lemma uses in an essential way that the basis point $q$ has trivial isotropy. At points with non-trivial isotropy one would not expect there to be a set of equivariant vector fields that span $D$.

The following theorem shows that, for an adapted basis $\{Z_b, Y_\alpha\}$, the coefficients 
\begin{equation*}
C_{\alpha \beta \gamma}:=\langle [X_\alpha , X_\beta ] , X_\gamma \rangle
\end{equation*}
satisfy
\begin{equation}
\label{E:GaugeSkew}
C_{b\alpha \beta}= - C_{b \beta \alpha}.
\end{equation}
This is central to the definition of the 3-form $\Lambda$ in Lemma \ref{L:Global3form} below. Moreover, this property serves to characterise gauge momenta.

\begin{theorem}
\label{Th:skew-gauge-mom} 
Let $\{X_\alpha\}$ ($\alpha=1,\dots,r=\dim(D)$) be globally defined equivariant vector fields on $Q$ which on $Q_f$ generate $D$.  Suppose moreover that $X_1$ is a section of $\mathcal{S}$. Then $X_1$ generates a gauge momentum if and only if, for all $\alpha,\beta$,
\begin{equation*}
C_{1\alpha \beta}= - C_{1\beta \alpha}.
\end{equation*}
\end{theorem}

\begin{proof}
Consider the quantities $F_{\alpha\beta}=\left<X_\alpha,X_\beta\right>$.
Since the $X_\alpha$ are equivariant and the metric is $G$-invarant, the $F_{\alpha\beta}$ are invariant functions.  Since $X_1$ is tangent to group orbits, it follows that $X_1[F_{\alpha\beta}]=0$. Thus
\begin{eqnarray*}
	0&=&X_1\left [ F_{\alpha\beta} \right ]\\
	&=& \pounds_{X_1}(\mathcal{G})(X_\alpha,X_\beta)+\left<[X_1,X_\alpha],X_\beta\right>+\left<X_\alpha,[X_1,X_\beta]\right> \\
	&=& \pounds_{X_1}(\mathcal{G})(X_\alpha,X_\beta)+C_{1 \alpha \beta} + C_{1 \beta \alpha}.
\end{eqnarray*}
Here $\pounds_{X_1}(\mathcal{G})$ is the Lie derivative of the metric along $X_1$. Since $X_1$ is a section of $\mathcal{S}$ it follows that it annihilates the potential energy. Consequently, by Theorem \ref{Th:isometries on D}, $X_1$ generates a gauge momentum if and only if it is an infinitesimal isometry on $D$. Over $Q_f$ this latter condition is equivalent, by definition, to $\pounds_{X_1}(\mathcal{G})(X_\alpha,X_\beta)=0$ for all $\alpha,\beta=1,\dots,r$. This proves the theorem over $Q_f$.

Finally, note that if $p_{X_1}$ is a gauge momentum on $D^*$ restricted to $Q_f$, it is also one over all of $D^*$, since $Q_f$ is an open dense subset of $Q$.
\end{proof}

This theorem can be used to find gauge momenta in examples. See the treatment of the  solid of revolution that rolls without slipping
on the plane in Section \ref{SS:SolidRev}.

\subsection{The 3-forms $\Lambda$.} \label{SS:3form} From now on, we  assume that the subbundle $W$ of $TQ$ with the property 
$TQ=D\oplus W$ is $G$-invariant; this implies that the identification of $D^*$ with a subbundle of $T^*Q$ is equivariant.  A possibility  to achieve this is to take $W=D^\perp$ but, as
indicated before,  other choices may simplify coordinate calculations 
in concrete examples.

\begin{lemma}
\label{L:Global3form}
There exists a globally defined $G$-invariant 3-form $\Lambda$ on $Q_f$ with the following properties
\begin{enumerate}
\item $\Lambda$ vanishes upon contraction with  elements of $W$ (i.e. it can be interpreted as a section of $\wedge^3(D^*)$),
\item if   $\{X_\alpha\}= \{Z_b, Y_I\}$ is an adapted basis of sections of $D$ on   $Q_f$ then 
\begin{equation}
\label{Eq:Lambda coefficients-mod}
\Lambda(Z_b,X_\alpha,X_\beta)=\langle [Z_b,X_\alpha],X_\beta\rangle,
\end{equation}
for $1\leq b \leq \ell$, $1\leq \alpha , \beta \leq r$.
\end{enumerate}
\end{lemma}


\begin{remarkths}
\label{Rmk:ExistBasis}
\begin{enumerate}

\item It follows from Theorem \ref{Th:skew-gauge-mom} that condition (ii) in the above Lemma is consistent with the antisymmetry relation
$\Lambda(Z_b,X_\alpha,X_\beta)=-\Lambda(Z_b,X_\beta,X_\alpha)$. On the other hand, in general,  $\Lambda(X_\alpha,X_\beta, Z_b)$ does not equal $\left<[X_\alpha,X_\beta],\, Z_b\right>$.

\item Conditions (i) and (ii) of the lemma impose no restriction on the values of  $\Lambda( Y_I,Y_J,Y_K)$ which leads to the  possible non-uniqueness of $\Lambda$. However,
if $r-\ell<3$,  given that  the indices $I,J,K$ run from $1$ to $r-\ell$, then $\Lambda( Y_I,Y_J,Y_K)=0$ and 
formula \eqref{Eq:Lambda coefficients-mod} completely characterises $\Lambda$. Therefore, in this case  Lemma \ref{L:Global3form} defines a unique 3-form on $Q_f$.

\item Equation \eqref{Eq:Lambda coefficients-mod} bears a similarity with the definition of the Cartan 3-form on Lie groups. However the vector fields involved are only equivariant and not necessarily tangent to the group orbit.

\item 
It is worth pointing out that while the construction of $\Lambda$ depends explicitly on the kinetic energy (metric), it does not depend on the potential part of the Lagrangian provided it is $G$-invariant. This $G$-invariance is required because the vector fields $Z_b$ must annihilate the potential (see Theorem \ref{Th:isometries on D}). 
\end{enumerate}
\end{remarkths}

The proof of Lemma \ref{L:Global3form} can be found in the Appendix (Proposition \ref{prop:global consistency}).  Note that the local expression
for a 3-form $\Lambda$ satisfying the conditions in the lemma  may be given
in terms of an adapted basis of sections $\{X_\alpha\}= \{Z_b, Y_I\}$ of $D$  by 
\begin{equation}
\label{E:Casimir3form}
\Lambda=\tfrac{1}{6}B_{\alpha \beta \gamma} \, \mu^\alpha \wedge \mu^\beta \wedge \mu^\gamma,
\end{equation}
where the coefficients $B_{\alpha \beta \gamma}$ are $G$-invariant functions, alternating in the indices, that satisfy
\begin{equation}\label{Eq:Lambda coefficients}
B_{b\beta\gamma} = \left<[Z_b,X_\beta],\, X_\gamma\right>.
\end{equation}
As usual, in \eqref{E:Casimir3form}, $\{\mu^\alpha\}$ is a basis of sections of $D^*=W^\circ$ that is dual to $\{X_\alpha\}$. Namely, they
are locally defined 1-forms on $Q$ that annihilate $W$ and satisfy $\mu^\alpha(X_\beta)=\delta^\alpha_\beta$. 
The $G$-invariance of  $W$   guarantees that $\mu^\alpha$ are also  $G$-invariant.

\subsection{Almost Poisson brackets having gauge momenta as Casimirs} \label{SS:FinalCasimirs}
We assume for the remainder of this section that among the 3-forms of Lemma \ref{L:Global3form}, there exists at least   one 
that admits a smooth extension to the points of  $Q$ having non-trivial isotropy. We point out that  this property is satisfied in 
 all of the examples that we considered. Denote the resulting  3-form on $Q$ by $\Lambda$.

We now use this 3-form $\Lambda$ to apply 
the construction outlined in Section \ref{S:GaugeT} to construct a bracket $\Pi_{nh}^\Lambda$ for our
nonholonomic system to obtain our main result.

\begin{theorem}
\label{th:main}
The bivector $\Pi_{nh}^\Lambda$ on $D^*$ is $G$-invariant and the gauge momenta $\pi_b$, $b=1,\dots, \ell$, are Casimir
functions of the induced bracket on the reduced space $D^*/G$.  
\end{theorem}

\begin{proof}
That $\Pi_{nh}^\Lambda$  is $G$-invariant follows from the invariance of the 3-form $\Lambda$ and of the bivector $\Pi_{nh}$.   For the proof that the gauge momenta are Casimir functions, we first prove this on $Q_f$ and then deduce the full statement by continuity.  

We may obtain a  local expression for $\Pi_{nh}^\Lambda$  on $Q_f$ using formula \eqref{E:PiBCoords}  in Theorem \ref{T:PropPiB}. Recall that $B_{\alpha \beta}^\gamma:=\mathcal{G}^{\gamma \delta} B_{\alpha \beta \delta}$, where  $B_{\alpha \beta \delta}=\Lambda(X_\alpha,X_\beta,X_\delta)$ as indicated in \eqref{eq:Bcoeff-general}. Using \eqref{Eq:Lambda coefficients}
and  \eqref{E:defCalbega} we find
\begin{equation*}
\begin{split}
B_{b\beta}^\gamma=C_{b\beta}^\gamma, \qquad B_{\beta b}^\gamma=C_{\beta b}^\gamma, \qquad
B_{IJ}^b=\mathcal{G}^{bc}C_{cIJ}+\mathcal{G}^{bK}B_{IJK}, \qquad B_{IJ}^K=\mathcal{G}^{Kb}C_{bIJ}+\mathcal{G}^{KL}B_{IJL}.
\end{split}
\end{equation*}
Therefore, the expression for $\Pi_{nh}^\Lambda$ becomes
\begin{equation}
\label{eq:Pi-Lambda-Casimir-General}
\Pi_{nh}^\Lambda=\rho^i_\alpha \partial_{q^i}\wedge \partial_{\pi_\alpha} +\frac 12  \left ( \left (\mathcal{G}^{bc}C_{cIJ}+\mathcal{G}^{bK}B_{IJK} - C_{IJ}^b \right )\pi_b + \left ( \mathcal{G}^{Kb}C_{bIJ}+\mathcal{G}^{KL}B_{IJL} -C_{IJ}^K \right ) \pi_K \right ) \, \partial_{\pi_I}\wedge \partial_{\pi_J} .
\end{equation}

A direct calculation using the above formulae for $\Pi_{nh}^\Lambda$ gives
\begin{equation*}
(\Pi_{nh}^\Lambda)^\sharp(d\pi_b)=- \rho_b^i\partial_{q^i}.
\end{equation*}
We claim that this vector field on $D^*$ is tangent to the group orbits of the lifted action of $G$ to $D^*$. Indeed, given that $Z_b= \rho_b^i\partial_{q^i}$ is
a section of $\mathcal{S}$,
for any $q\in Q_f$   there exists a Lie algebra element $\xi (q)\in \g$ such that
$Z_b(q)$ coincides with the infinitesimal generator  of $\xi (q)$  at $q$. Namely, 
\begin{equation*}
Z_b(q)= \left . \frac{d}{ds} \right |_{s=0} \exp(s\xi (q))\cdot q.
\end{equation*}
Now, since all vector fields in the  basis $\{X_\alpha\}= \{Z_b, Y_I\}$ are equivariant, the corresponding
momenta $\pi_\alpha$ are invariant functions. Therefore,  in coordinates $(q^i, \pi_\alpha)$, the local expression for the infinitesimal generator of $\xi (q)$ of the lifted action of $G$ to $D^*$ is
\begin{equation*}
 \left . \frac{d}{ds} \right |_{s=0} \exp(s\xi (q))\cdot (q,\pi)=\rho_b^i(q)\partial_{q^i}, 
\end{equation*}
which establishes the claim.  

It follows that the  Hamiltonian vector
field associated to $\pi_b$ via the induced bracket on the reduced space vanishes; i.e. $\pi_b$ is a Casimir of the reduced bracket, over all points of the open dense subset $Q_f$.  It then follows by continuity that $\pi_b$ (which is a globally defined invariant smooth function) is a Casimir everywhere. 
\end{proof}

\begin{remarkth}
We stress that the theorem assumes that among the 3-forms of Lemma \ref{L:Global3form}, there exists at least   one 
that admits a smooth extension to the points of  $Q$ having non-trivial isotropy, and that   this property is satisfied in 
all of the examples that we considered. However, we do not have a proof that such an extension always exists.  One may of course restrict the construction to $Q_f$, but the differential equation may then fail to be complete.
\end{remarkth}

 We conclude this section with some dynamical consequences of the above theorem. For free actions, ``directions" that annihilate the group orbits are spanned by the pull-backs to $D^*$ of 1-forms on $D^*/G$.  For more general actions, one replaces such forms by {\em basic} 1-forms, which are those $G$-invariant 1-forms $\beta$ on $D^*$ satisfying $\beta(u)=0$ for all $u$ tangent to the group orbit. Note that the differential of any invariant function is basic.

\begin{corollary}
\label{C:Rank2}
Suppose that the generic level sets of the gauge momenta on the reduced space $D^*/G$  are 2-dimensional. Then
\begin{enumerate}
\item the reduced bracket defined on $D^*/G$ induced by  $\Pi_{nh}^\Lambda$  satisfies the Jacobi identity and the system is Hamiltonizable,
\item the flow of the system on $D^*$  preserves the restriction of $(\Pi_{nh}^\Lambda)^\sharp$ to directions that annihilate the group orbits. More precisely,
\begin{equation*}
\left(\pounds_{X_{nh}}(\Pi_{nh}^\Lambda)\right)^\sharp ( \beta)=0
\end{equation*}
for every basic 1-form $\beta$ on $D^*$. In particular this holds when $\beta=d\pi_b$, for any gauge momentum $\pi_b$ and therefore, for each $b=1,\dots,\ell$,
$$[X_{nh},\, X_{\pi_b}]=0,$$
where $X_{\pi_b}=(\Pi_{nh}^\Lambda)^\sharp (d\pi_b)$.
\end{enumerate}
\end{corollary}

\begin{proof}
(i) The proof of this is simply that on any 2-dimensional manifold, every almost-Poisson structure is in fact Poisson (as the Schouten-Nijenhuis bracket of the structure with itself, which is an alternating 3-tensor, must vanish).

(ii) The proof proceeds stratum by stratum: the flow of the vector field $X_{nh}$ preserves the strata because it is equivariant.  On the stratum where the action is free, any basic 1-form $\beta$ is the pull-back $\beta=\tau^*\alpha$ for some 1-form $\alpha$ on $(D^*/G)_f$, where $\tau:D^*\to D^*/G$ is the projection. Then,
$$\pounds_{X_{nh}}(\Pi_{nh}^\Lambda)^\sharp ( \beta) = \pounds_{X_{nh}}(\Pi_{nh}^\Lambda)^\sharp ( \tau^*\alpha) =  \left(\tau_*\left(\pounds_{X_{nh}}(\Pi_{nh}^\Lambda)\right)\right)^\sharp ( \alpha)
$$
and by the natural properties of Lie derivatives, this is
$$\left(\pounds_{\tau_*X_{nh}}(\tau_*\Pi_{nh}^\Lambda)\right)^\sharp ( \alpha)$$
and this vanishes by part (i), and the fact that Hamiltonian flows are Poisson.

The argument for other strata proceeds in the same way, since the restriction of $\tau$ to an orbit-type stratum is a submersion \cite{Duistermaat-Kolk}. 
\end{proof}

\subsection{Final observations}
\label{SS:FinalObservations}

\subsubsection{The case $r-\ell<3$} If the difference between the rank of the distribution $r$ and the number of independent gauge momenta $\ell$ is less than 3, then
as indicated in  Remark~\ref{Rmk:ExistBasis}(ii), the 3-form $\Lambda$  of Lemma  \ref{L:Global3form} is uniquely determined (given the choice of $W$) and, therefore, the same is true for the bracket $\Pi_{nh}^\Lambda$ appearing in 
Theorem~\ref{th:main}. 
In this case, since the indices $I,J,K$ run from $1$ to $r-\ell$, all coefficients  $B_{IJK}=0$ by the alternating property, and \eqref{eq:Pi-Lambda-Casimir-General} simplifies to 
\begin{equation}
\label{E:PinhB-simple}
\Pi_{nh}^\Lambda=\rho^i_\alpha \partial_{q^i}\wedge \partial_{\pi_\alpha} +\frac 12  \left (\mathcal{G}^{\gamma b}C_{bIJ} - C_{IJ}^\gamma \right )\pi_\gamma  \, \partial_{\pi_I}\wedge \partial_{\pi_J} .
\end{equation}
This formula gives an expression for the bracket whose invariant gauge momenta pass to the reduced space as Casimir functions  for  a number of classical nonholonomic problems, such
as the nonholonomic particle, the Chaplygin sphere, the problem of solids of revolution that roll without slipping on a horizontal plane or on a sphere, 
a homogeneous ball that rolls without slipping on a 
surface of revolution, and also the class of rigid bodies subject to generalised rolling constraints introduced in \cite{BalseiroGN}. 

In order to use the above formula in examples one needs to obtain an expression for an adapted basis of sections to compute  the coefficients $\rho^i_\alpha, \, C_{\alpha \beta}^\gamma, \, 
\mathcal{G}^{\alpha \beta}$. In sections \ref{S:ChSphere} and \ref{SS:SolidRev} we illustrate how this is done for the Chaplygin  sphere and for a solid of revolution that rolls without slipping on the plane. For these examples we also compute the 3-form $\Lambda$, both of which are given by interesting expressions involving the  Cartan 3-form on $SO(3)$.

\subsubsection{Why do we need to modify the brackets?} \label{sec:why gauge?} 
We now address the following fundamental question: why does one need to modify the nonholonomic bracket $\Pi_{nh}$
by $\Lambda$  to guarantee that the invariant gauge momenta
are Casimir functions of the reduced bracket?

  First we note that the question is only relevant  if the rank $r$ of the constraint distribution 
$D$ is greater than 2. Indeed, if $r\leq 2$ then the section $\Lambda$ of $\wedge^3(D^*)$ of Lemma  \ref{L:Global3form} vanishes and the bracket  
 $\Pi_{nh}^\Lambda$ in Theorem \ref{th:main}
coincides with $\Pi_{nh}$. An instance of this general situation is  encountered  by Balseiro  \cite{Balseiro2016}  in her treatment of the nonholonomic particle.

Consider then the case where $r>2$,  and assume for simplicity that there is only one (invariant) gauge momentum  whose unique generator taking values in $D$ is the 
equivariant vector field $Z$.
A fundamental property of $Z$, that  contrasts with the  
situation encountered in holonomic systems, is that even though $Z$ is tangent to the 
group orbits, it may not coincide with the infinitesimal generator  $\xi_Q$ of a constant Lie algebra element $\xi \in \g$ \cite{Bates}. It is this property of $Z$, that is not encountered in holonomic mechanics,
 that leads to the need of the modification of the nonholonomic 
bracket. Indeed, if $Z$ were equal to $\xi_Q$ for a fixed   $\xi \in \g$, and $\{Z,Y_I\}$ is an adapted basis of sections of $D$, then $[Z,Y_I]=0$ by equivariance of $Y_I$. Therefore
the 3-form $\Lambda=0$ satisfies the conditions of Lemma  \ref{L:Global3form} and the corresponding bracket $\Pi_{nh}^\Lambda$ in Theorem \ref{th:main}
coincides with $\Pi_{nh}$.

Examples of nonholonomic systems where $Z=\xi_Q$  for a fixed $\xi \in \g$ are rare and one usually must perform the modification of $\Pi_{nh}$ developed above to guarantee that the invariant gauge momenta pass to the quotient space as Casimir functions. An (somewhat artificial) example  where this condition holds and no modification of the nonholonomic bracket
$\Pi_{nh}$ is needed to accomplish the aforementioned goal, is
the ``rank 1" case of a rigid body subject to generalised rolling constraints treated in \cite{BalseiroGN}.

\subsubsection{Future work} As mentioned in the introduction, it would be interesting to extend the results of our paper to the case where the gauge momenta of a nonholonomic
system are not $G$-invariant. In this case it may be possible to find modifications of the nonholonomic bracket to guarantee that the $G$-invariant (nonlinear) 
functions of the gauge momenta pass to the quotient space as Casimir functions. This seems to be possible to accomplish for some examples,  and in particular for  the nonholonomic hinge \cite{Biz}.

\section{The  Chaplygin sphere}
\label{S:ChSphere}

This problem was considered by Chaplygin in \cite{chapsphere} and concerns the motion of an inhomogeneous sphere, whose center of mass coincides with its geometric centre, which is the rolling without slipping on 
a fixed plane. The vertical component of the angular momentum of the sphere is a gauge momentum for this problem. 
Using our method it is possible to construct an almost Poisson  bracket for the system that upon reduction has this angular momentum as Casimir. This bracket was
first found by Borisov and Mamaev \cite{BorMamChap}.

We assume that two of the moments of inertia of the sphere coincide to simplify the algebra and better illustrate how our construction works but a similar approach works in the general case.

The configuration space for the problem is $Q=SO(3)\times \R^2$. The attitude matrix $\mathcal{R}\in SO(3)$ specifies the orientation of the sphere by 
relating a body frame centred at the centre of the sphere, with a space frame whose third axis is perpendicular to the fixed plane. 
Let  $(x,y)\in \R^2$ be the spatial coordinates of the contact point of the sphere with the plane.

We will use Euler angles as local
coordinates for 
$SO(3)$. In accordance with  the {\em $x$-convention}, see e.g. \cite{MaRa}, we write a matrix $\mathcal{R}\in SO(3)$ 
as 
\begin{equation}
\label{E:MatR}
\mathcal{R}=\left(
\begin{array}{ccc}
 \cos \psi \cos \varphi - \cos \theta \sin \varphi \sin \psi & -\sin \psi \cos \varphi - \cos \theta \sin \varphi \cos \psi & \sin \theta \sin \varphi \\
\cos \psi \sin \varphi + \cos \theta \cos \varphi  \sin \psi  & -\sin \psi \sin \varphi + \cos \theta \cos \varphi  \cos \psi  & -\sin \theta \cos \varphi  \\
 \sin \theta \sin \psi   & \sin \theta \cos \psi  & \cos \theta  
\end{array}
\right),
\end{equation}
where the Euler angles $0<\varphi , \psi <2\pi, \, 0<\theta <\pi$. According to this convention, we
obtain the following expressions for the angular velocity in space coordinates $ {\omega}$,
and in body coordinates ${\Omega}$ (see e.g. \cite{MaRa}):
\begin{equation}
\label{E:Omega-omega}
 \omega=\left (\begin{array}{c} \dot \theta \cos \varphi +\dot \psi \sin \varphi \sin \theta \\ \dot \theta \sin \varphi - \dot \psi \cos \varphi \sin \theta  \\ \dot 
\varphi + \dot \psi \cos \theta \end{array} \right ), \qquad  \Omega=\left ( \begin{array}{c} \dot \theta \cos \psi +\dot \varphi \sin \psi \sin \theta \\ -\dot \theta \sin \psi + \dot \varphi \cos \psi \sin \theta  \\ \dot 
\varphi\cos \theta + \dot \psi  \end{array} \right ).
\end{equation}
 The constraints of rolling without slipping are
\begin{equation}
\label{E:Constraints-rolling-Chap-top}
\dot x= R\omega_2 =R (\dot \theta \sin \varphi - \dot \psi \cos \varphi \sin \theta ), \qquad \dot y = -R\omega_1
=-R(\dot \theta \cos \varphi +\dot \psi \sin \varphi \sin \theta),
\end{equation}
where $R$ is the radius of the sphere.

 Assuming that the third axis of the body frame is  the axis of symmetry of the sphere,  the
 Lagrangian is 
\begin{equation*}
L=\frac{1}{2}\left (I_1(\dot \theta \cos \psi +\dot \varphi \sin \psi \sin \theta)^2 + I_1(-\dot \theta \sin \psi + \dot \varphi \cos \psi \sin \theta)^2 +I_3(\dot \varphi\cos \theta + \dot \psi)^2 +m(\dot x^2+\dot y^2)\right )-V(\theta, \psi)
\end{equation*}
where $I_1,I_1,I_3$ are the principal moments of inertia and $m$ is the total mass of the sphere. Here $V$ is a potential energy chosen
to be invariant under the symmetry group action defined below. Expanding the above expression, one finds that the kinetic energy metric 
is 
\begin{equation*}
\mathcal{G}=(I_1\sin^2\theta+I_3\cos^2\theta)\, d\varphi^2+I_1\, d\theta^2+I_3\, d\psi^2 +m \, dx^2 + m\, dy^2 +2 I_3\cos \theta \, d\varphi \, d\psi.
\end{equation*}

The symmetry group is $G=SE(2)$. The action of $(\vartheta, a,b)\in SE(2)$ on $Q$ is free and proper and is given in the above coordinates
by
\begin{equation*}
(\vartheta, a,b)\,:\,(\varphi,\theta,\psi,x,y)\mapsto (\varphi+\vartheta,\theta,\psi,x \cos \vartheta -y \sin \vartheta+a ,
y \cos \vartheta +x \sin \vartheta +b).
\end{equation*}

One checks that both the constraints and the Lagrangian are invariant under the lift of this action to $TQ$.
The constraint distribution $D$ has rank 3 and is spanned by the $G$-equivariant vector fields
\begin{equation}
\label{E:VFChap}
Z_1=\partial_\varphi, \qquad Y_2=\partial_\theta +R\sin \varphi \,\partial_x -R\cos \varphi \,\partial_y, \qquad 
Y_3=\partial_\psi -R\cos \varphi\sin \theta\, \partial_x -R\sin \varphi\sin \theta\,\partial_y.
\end{equation}
We note that the basis of sections $\{X_\alpha \}=\{Z_1, Y_2, Y_3\}$ is adapted in the sense of Definition \ref{E:DefAdaptedBasis}. Indeed, its elements
are equivariant vector fields and $Z_1$ generates a gauge momentum. This can be seen by noticing that $Z_1$ is tangent to the
group orbits and that its  tangent lift 
in bundle coordinates is $Z_1^{TQ}=\partial_\varphi$. Hence $Z_1^{TQ}[L]=0$ and therefore $\dot p_{Z_1}=0$ by \eqref{E:pZdot}.

The non-zero coefficients $C_{\alpha \beta \gamma}$ (with $\alpha<\beta$) are computed to be
\begin{equation}
\label{E:CdownChap}
C_{123}=-mR^2 \sin \theta, \qquad C_{132} =mR^2 \sin \theta, \qquad C_{233}=mR^2 \sin \theta \cos \theta.
\end{equation}
All other terms may be determined by the skew-symmetry $C_{\alpha \beta \gamma} = -C_{ \beta \alpha \gamma}$. Note that $C_{123}=-C_{132}$ which serves as a double check that $p_{Z_1}$ is 
a gauge momentum in view of Theorem \ref{Th:skew-gauge-mom}.

We select the dual 1-forms $\{ \mu^\alpha \} = \{ d\varphi , d\theta , d\psi \}$ which amounts to identifying 
$D^*=(\mathrm{span}\{\partial_x , \partial_y\})^\circ =T^*SO(3)\times \R^2$.
This is allowed in our construction since $W=\mathrm{span}\{\partial_x , \partial_y\}\subset TQ$ is invariant under the $SE(2)$ action.
The formula  \eqref{E:Casimir3form} defines a  unique 3-form $\Lambda$ since the rank of $D$ is $r=3$  and we have  $l=1$ gauge momentum, and so $r-l<3$.
Such unique 3-form  is  given by
\begin{eqnarray*}
\Lambda=-mR^2 \sin \theta \, d\varphi \wedge d\theta \wedge d\psi.
\end{eqnarray*}

A coordinate independent expression  for $\Lambda$ may be given in terms of the unique left-invariant 1-forms $\lambda^1, \lambda^2, \lambda^3$ on $SO(3)$
that at the group identity are dual to the canonical basis of $\mathfrak{so}(3)\cong \R^3$. These 1-forms have local expressions:
\begin{equation}
\label{E:Def1formlambda}
\lambda^1=\sin \psi \sin \theta \, d\varphi+ \cos \psi  \, d\theta, \quad 
\lambda^2=\cos \psi \sin \theta \, d\varphi- \sin \psi \, d\theta, \quad \lambda^3=\cos \theta \, d\varphi +d\psi.
\end{equation}
So we can write
\begin{equation*}
\Lambda=mR^2 \,\lambda^1\wedge \lambda^2 \wedge \lambda^3.
\end{equation*}
Therefore, up to the constant factor of $mR^2$,  $\Lambda$ equals  the Cartan bi-invariant volume form on $SO(3)$ normalised to have volume one on the unit cube of $\mathfrak{so}(3)\cong \R^3$. This also holds for the general Chaplygin sphere with arbitrary moments of inertia as had been indicated in \cite{LGN10}.

An expression for the bracket $\{\cdot ,\cdot \}$ defined by $\Pi_{nh}^\Lambda$, whose reduction has $\pi_1$ as a Casimir, and describes the nonholonomic dynamics,
 can  be obtained using \eqref{E:PinhB-simple}. 
All the brackets between the coordinate functions $\varphi, \theta, \psi, x, y$ are zero. The non-zero brackets between the coordinates and the momenta $\pi_\alpha$ are
obtained using \eqref{E:VFChap}:
\begin{equation*}
\begin{split}
\{\varphi, \pi_1\}=1, \qquad \{\theta, \pi_2\}=1, \qquad \{x, \pi_2 \} =R\sin \varphi, \qquad \{y, \pi_2 \} =- R\cos \varphi, \\
 \{\psi, \pi_3\}=1, \qquad \{x, \pi_3 \} =-R\cos \varphi \sin \theta, \qquad \{y, \pi_3 \} =- R\sin \varphi \sin \theta. 
\end{split}
\end{equation*}
Also, we have 
\begin{equation*}
\{\pi_1, \pi_2\}=0, \qquad \{\pi_1, \pi_3\}=0,
\end{equation*}
that must hold since $\pi_2$ and $\pi_3$ are invariant functions and $\pi_1$ is a Casimir of the reduced bracket.
To compute $ \{\pi_2, \pi_3\}$ one needs  to compute the inverse of the $3\times 3$ symmetric matrix $\mathcal{G}$ with entries $\mathcal{G}_{\alpha \beta}=\langle X_\alpha, X_\beta\rangle$.  One gets
\begin{equation}
\label{E:GinvChap}
\begin{split}
\mathcal{G}^{-1}=\frac{1}{K(\theta)\sin^2\theta }\begin{pmatrix}  I_3+mr^2\sin^2\theta & 0 & -I_3\cos \theta \\ 0 & \frac{K(\theta)\sin^2\theta}{ I_1+mr^2} & 0 
\\  -I_3\cos \theta & 0 &  I_1\sin^2\theta+I_3\cos^2\theta   \end{pmatrix}
\end{split}
\end{equation}
where 
\begin{equation*}
K(\theta)= I_1 mR^2 \sin^2 \theta +I_3 mR^2 \cos^2 \theta +I_1I_3.
\end{equation*}
With the aid of these formulae and \eqref{E:CdownChap} one obtains
\begin{equation*}
\begin{split}
C_{23}^1=-\frac{mR^2 I_3  \cos^2 \theta}{K(\theta) \sin \theta}, \qquad C_{23}^2=0, \qquad C_{23}^3 = \frac{mR^2\cos \theta (I_1\sin^2\theta +I_3\cos^2 \theta)}{K(\theta)\sin \theta},
\end{split}
\end{equation*}
and hence, by \eqref{E:PinhB-simple} we get
\begin{equation*}
\{\pi_2,\pi_3\}=-\frac{mR^2(I_3+mR^2)\sin \theta}{K(\theta)}\pi_1 - \frac{mR^2(I_1-I_3)\cos\theta\sin\theta}{K(\theta)}\pi_3.
\end{equation*}

Let us now write the bracket in terms of more standard physical variables for the problem. First, we introduce the Poisson vector $\gamma:=R^Te_z$ that gives the  coordinates on 
the body frame
of the vector normal to the plane on which the rolling takes place. Its components are
 \begin{equation}
 \label{E:gamma-angles}
\gamma_1=\sin \theta \sin \psi, \qquad \gamma_2= \sin \theta \cos \psi, \qquad \gamma_3=\cos \theta.
\end{equation}
Next, the angular momentum vector about the contact point, expressed in the body frame is given by 
\begin{equation*}
M=\I \Omega + mR^2 \gamma\times (\Omega \times \gamma),
\end{equation*}
where $\I=diag(I_1,I_1,I_3)$ is the tensor of inertia, and $\times$ denotes the vector product in $\R^3$. To write $M$ in terms of 
our coordinates and the momenta $\pi_\alpha$ start by noticing the quasi-velocities $v^\alpha$ defined by the basis $\{Z_1, Y_2,Y_3\}$
given by \eqref{E:VFChap} satisfy
\begin{equation}
\label{E:Quasi-Chap}
\begin{split}
\dot \varphi=v^1, \qquad \dot \theta=v^2, \qquad \dot \psi =v^3.
\end{split}
\end{equation}
Next, write $v^\alpha=\mathcal{G}^{\alpha \beta}\pi_\beta$ using the expression for $\mathcal{G}^{-1}$ given above in \eqref{E:GinvChap}. Combining
this with the expression for $\Omega$ given in \eqref{E:Omega-omega} and the expression for $\gamma$ on \eqref{E:gamma-angles} one gets
\begin{equation*}
M_1=\frac{\sin \psi \pi_1 + \cos \psi \sin \theta \pi_2 - \sin \psi \cos \theta \pi_3}{\sin \theta}, \qquad M_2=\frac{\cos \psi \pi_1 - \sin \psi \sin \theta \pi_2 - \cos \psi \cos \theta \pi_3}{\sin \theta}, 
\end{equation*}
 and $M_3=\pi_3$.
Note that both  vectors $M$ and $\gamma$ are $SE(2)$-invariant and its components drop down to the quotient $D^*/SE(2)$.  In fact, as a manifold  $D^*/SE(2)=\R^3\times S^2$.
The entries of $M$ serve as coordinates on the $\R^3$ factor while
 the components of $\gamma$ are redundant coordinates on $S^2$. In particular notice that $\pi_1$ equals the vertical component of the
 angular momentum vector; i.e.  $\pi_1=( M , \gamma )$, where $(\cdot  ,\cdot )$ is the euclidean scalar product in $\R^3$
 
By direct calculation using the above formulae for the bracket one gets
\begin{equation}
\label{E:brackChap}
\{\gamma_i, \gamma_j\}=0,  \qquad \{M_i, \gamma_j \}=-\epsilon_{ijk}\gamma_k, \qquad \{M_i , M_j\}=-\epsilon_{ijk}(M_k -mR^2( \Omega , \gamma ) \gamma_k),
\end{equation}
where the alternating tensor $\epsilon_{ijk}$ equals $0$ if two of its indices are equal, it equals $1$ if $(i,j,k)$ is a cyclic permutation of $(1,2,3)$ and it equals $-1$ otherwise.
The term $( \Omega , \gamma)$ is the spinning speed  of the ball about the vertical axis. It can be written   in terms of $M$ and $\gamma$ using the 
expression
\begin{equation*}
\Omega = AM+ mR^2 \frac{(AM, \gamma)}{1-mR^2(A\gamma, \gamma)}A\gamma,
\end{equation*}
where $A=(\I + mR^2)^{-1}$.
(For the calculations it is useful to note that $K(\theta)=(I_1+mR^2)(I_3+mR^2)(1-( A\gamma ,  \gamma ))$).

The above formulae determine the reduced bracket in the quotient space $D^*/SE(2)$. The reduced equations of motion (see e.g. \cite{FedKoz})
\begin{equation*}
\dot M= M\times \Omega, \qquad \dot \gamma = \gamma \times \Omega,
\end{equation*}
are Hamiltonian with respect to $H=\frac{1}{2}(M, \Omega)+ V(\gamma)$.

The bracket \eqref{E:brackChap} has rank 4 and, even though it does not satisfy the Jacobi identity, its characteristic
distribution is integrable -  the leaves of the corresponding foliation are the level sets of the Casimir function $( M , \gamma )$.
As was first  noticed in \cite{BorMamChap}, the bracket obtained by multiplication by the conformal factor $\sqrt{1-mR^2( A\gamma ,  \gamma ) }$  does satisfy the Jacobi identity and Hamiltonizes the problem. This kind of multiplication is commonly interpreted as a {\em time reparametrisation}. We mention that the conformal factor
is intimately related with the preserved measure of the problem \cite{FedKoz}.

\section{A solid of revolution rolling without slipping on a fixed plane} \label{SS:SolidRev}

We consider a convex body of revolution with a smooth surface that rolls without slipping on a fixed plane.
This problem was originally considered by Routh \cite{Routh} in the case of a spherical body, and by Chaplygin \cite{Ch_r} and Appel \cite{Appell} in the general case; see  \cite{BorMam}
for historical details. Our treatment and notation is close to the one used in  \cite{BorMam}.

 As for the Chaplygin sphere, the configuration space is  $Q=SO(3)\times \R^2$.  We 
 denote by  $u=(x,y,z)$ 
  the  coordinates of the centre of mass $O$ of the body with respect to an inertial frame. Our choice of inertial frame is such that the fixed plane where the rolling takes place  corresponds to $z=0$.
  On the other hand, the body frame is chosen to be centred at the centre of mass, and having third axis $E_3$ along the symmetry axis of the body.
   
   Denote by $\mathcal{R}\in SO(3)$ 
 the attitude matrix that relates the two frames.  The constraints of rolling without slipping are  given by
 \begin{equation}
 \label{E:RollingCBodyREv}
\dot u=\dot {\mathcal{R}}\rho,
\end{equation}
where $\rho$ is the vector from contact point $P$ to the centre of mass of the body $O$  written in the body frame (see
Figure \ref{F:planarsection}).
The last component in the above equation is in fact the holonomic constraint
\begin{equation}
\label{E:zBodyRev}
z=\langle \rho, \gamma \rangle,
\end{equation}
where, just like in the previous section, $\gamma = \mathcal{R}^Te_z$ is the vector normal to the fixed plane written in body coordinates. We assume that the orientation of $e_z$
is such that  that $\gamma$ is the inward normal vector of the body at $P$. Hence, the inverse of the classical Gauss map from differential geometry
of surfaces, allows us to express  $\rho$  as a function of $\gamma$ in the form
\begin{equation*}
\rho_1=f_1(\gamma_3)\gamma_1, \qquad \rho_2=f_1(\gamma_3)\gamma_2, \qquad \rho_3=f_2(\gamma_3).
\end{equation*}

By writing  $\gamma_3=\cos \theta$ in accordance with the Euler angles introduced in Section \ref{S:ChSphere}, then we may write $f_1, \, f_2$ as functions of $\theta$.
Their geometric
meaning can be read off from Figure \ref{F:planarsection} that depicts the curve that generates the surface of revolution in the perspective of the body frame. Notice that
 $\pi-\theta$ is the angle
between the $E_3$ axis and the outer normal vector to the surface and 
 $a_1(\theta):=f_1(\theta)\sin \theta$ is the distance between $P$ and the $E_3$ axis. The figure also illustrates the height $z$ of the centre of mass.
Note that $(z,\theta - \pi/2)$ are polar coordinates with respect to $E_1$ for the {\em pedal curve} of the generating curve about $O$.
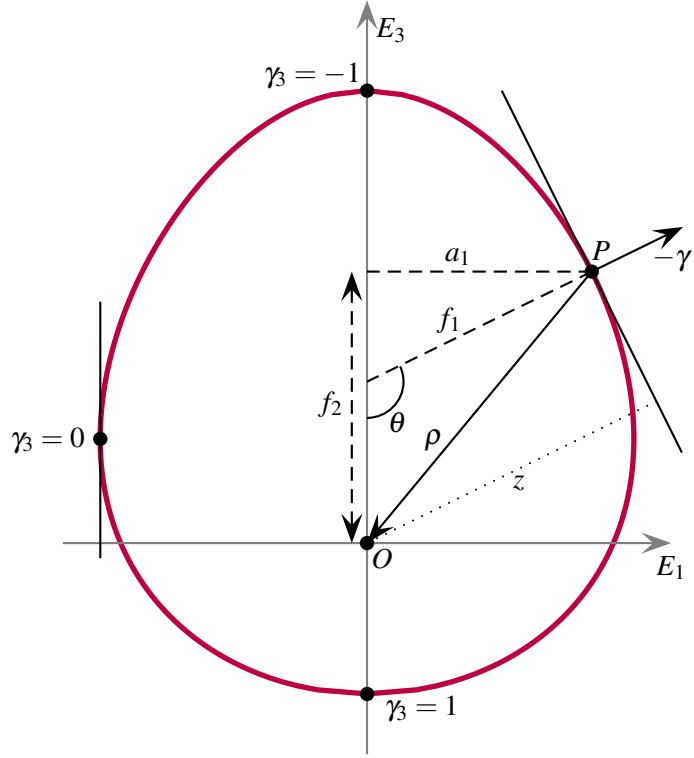
\begin{figure}[h]
\centering
\psset{unit=4,arrowsize=8pt,dotscale=1.5}
{\begin{pspicture}(-1,-1.2)(1,1.3)
\rput{-90}{
\psplot[plotpoints=150,linewidth=2pt,linecolor=purple]{-1}{1}{3 x add 1 x 2 exp sub mul sqrt 2 div}
\psplot[plotpoints=150,linewidth=2pt,linecolor=purple]{-1}{1}{3 x add 1 x 2 exp sub mul sqrt -2 div}
\psline[linecolor=Grey]{->}(1.2,0)(-1.3,0) 
\psline[linecolor=Grey]{->}(0.5,-1)(0.5,1) 
\psdots(1,0)(-1,0)
\psdot(-0.4, 0.739)
\psline(-1.0, 0.443)(0.2, 1.0353)
\psdot[dotstyle=*,dotscale=1.5](0.5,0)
\psline[arrowsize=8pt]{<-}(0.5, 0)(-0.4, 0.7389)
\psline[linestyle=dashed](-0.4, 0.739)(-0.4, 0)
\psline{->}(-0.4, 0.739)(-.548, 1.039)
\psline[linestyle=dashed](-0.4, 0.739)(-0.035, 0)
\psarc(-0.035, 0){0.12}{0}{115} 
\psline[linestyle=dashed]{<->}(0.5,-0.05)(-0.4,-0.05)
\psline[linestyle=dotted](0.5, 0)(0.03, 0.9513)
\psdot(.1547, -0.877)
\psline(0.55, -0.877)(-.3, -0.877)
}
\rput(0.08,1.2){$E_3$}  \rput(1,-0.58){$E_1$}
\rput(-0.18,1.05){$\gamma_3=-1$} \rput(0.18,-1.05){$\gamma_3=1$}
\rput(0.05,-0.55){$O$}  \rput(0.77,0.47){$P$}
\rput(0.3,0.45){$a_1$}   \rput(0.27,0.24){$f_1$} \rput(0.1,-0.1){$\theta$}
\rput(1,0.44){$-\gamma$} \rput(0.5,-0.3){$z$}
\rput(0.22,-0.16){$\rho$}
\rput(-0.12,-0.04){$f_2$}
\rput(-1.05,-0.16){$\gamma_3=0$} 
\end{pspicture}
}

\caption{The generator of a solid of revolution; see text for details}
\label{F:planarsection}
\end{figure}

Recall that the principal lines of curvature of a surface of revolution are the parallels  (the circles $\theta =ct$) and the meridians (perpendicular to the
parallels and having shape  equal to the rotating curve). All of the meridians meet at the {\em poles} where $\gamma_3=\pm 1$ ($\theta =0,\pi$). 
It is clear from Figure \ref{F:planarsection} that $f_1, f_2$ can be extended as even $2\pi$-periodic functions of $\theta$ that, consequently,
have local extrema at the poles. The radii of curvature along the parallels $R_p$ and the meridians $R_m$ are given by
\begin{equation*}
R_p =f_1, \qquad R_m = -\frac{1}{\sin \theta} \frac{d f_2}{d\theta} =\frac{d f_2}{d\gamma_3} .
\end{equation*}
Note that $R_p$ and $R_m$ are smooth, and positive and their values coincide at the poles. As functions of $\theta$, they are $2\pi$-periodic and even.

Now, using \eqref{E:gamma-angles} and \eqref{E:zBodyRev} we have $z(\theta)=\sin^2 \theta f_1(\theta) +\cos \theta f_2(\theta)$. In view of the identity
\begin{equation}
\sin \theta \frac{da_1}{d\theta} + \cos \theta \frac{df_2}{d\theta} = 0,
\end{equation}
that follows from the definition of $\theta$ in Figure \ref{F:planarsection}, we have
\begin{equation*}
a_2(\theta):=\frac{dz}{d\theta}=\sin \theta (\cos \theta f_1(\theta) -f_2(\theta)),
\end{equation*}
and we can write $\dot z = a_2(\theta)\dot \theta$ (which is consistent with \eqref{E:RollingCBodyREv}). The kinetic energy of the system is therefore
 \begin{equation*}
\mathcal{K}=\frac{1}{2}\langle \I\Omega ,\Omega \rangle + \frac{m}{2}(\dot x^2 +\dot y^2
+a_2(\theta)^2 \dot \theta^2).
\end{equation*}
The assumption that the body is axisymmetric implies that the inertia tensor
has the form $\I= \mathrm{diag}(I_1,I_1,I_3)$.
 In terms of Euler angles (using \eqref{E:Omega-omega}) we get the expression for the   Lagrangian
 \begin{equation}
 \label{E:LagSolidRev}
\begin{split}
L&=\frac{1}{2}\left ( (I_1\sin^2\theta +I_3\cos^2\theta)\dot \varphi^2 +(I_1+ma_2(\theta)^2)\dot \theta^2+I_3\dot \psi^2
+2I_3\cos \theta \dot \varphi\dot \psi  + m(\dot x^2+\dot y^2)\right )-V(\theta),
\end{split}
\end{equation}
 where  $m$ is the total mass of the body and the potential
 $V$ is an even function of $\theta$ that is
  invariant under the symmetries  of the system  that are  discussed ahead. If  the potential is gravitational then $V(\theta) = mgz(\theta)$.
 The constraints \eqref{E:RollingCBodyREv}  of rolling without slipping are expressed in coordinates as
\begin{equation*}
\begin{split}
\dot x = -a_2(\theta) \cos \varphi \, \dot \varphi  +z(\theta) \sin \varphi \, \dot \theta -  a_1(\theta) \cos \varphi\, \dot \psi \, ,
\qquad
 \dot y=- a_2(\theta) \sin \varphi \, \dot \varphi  - z(\theta) \cos \varphi \, \dot \theta -  a_1(\theta) \sin \varphi\, \dot \psi .
\end{split}
\end{equation*}
 
The symmetry group is $G=SE(2)\times S^1$ corresponding to translations and rotations on the rolling plane, and to the 
internal rotational symmetry of the body. This action is not free. The configurations for which the point of contact $P$ is one of the poles
have $S^1$ isotropy.  These configurations lie outside the Euler angle chart that has $0<\theta< \pi$. 
All points in our chart have trivial isotropy, and 
the action of $((\vartheta, a,b),\Psi)\in SE(2)\times S^1$ on $Q$ is represented 
by
\begin{equation*}
((\vartheta, a,b),\Psi)\,:\,(\varphi,\theta,\psi,x,y)\mapsto (\varphi+\vartheta,\theta,\psi+\Psi,x \cos \varphi -y \sin \varphi +a ,
y \cos \varphi +x \sin \varphi +b).
\end{equation*}
One checks that both the constraints and the Lagrangian are invariant under the lift of this action to $TQ$.

The constraint distribution has rank 3. Along the  points covered by our chart, it is spanned by the $G$-equivariant vector fields
\begin{equation*}
\begin{split}
W_1&=\partial_\varphi- a_2(\theta)( \cos \varphi \, \partial_x + \sin \varphi \, \partial_y), \\
  W_2&=\partial_\psi - a_1(\theta)( \cos \varphi \, \partial_x + \sin \varphi \, \partial_y), \\ 
Y_3&=\partial_\theta +z(\theta)(\sin \varphi \, \partial_x - \cos \varphi \, \partial_y).
\end{split}
\end{equation*}
Although the  vector fields $W_1$ and $W_2$ span the space $\mathcal{S}_q=D_q\cap \g\cdot q$, they generally do not generate gauge momenta
and hence $\{W_1, W_2, Y_3\}$ is not an adapted basis in the sense of Definition \ref{E:DefAdaptedBasis} (the vector fields $W_1,W_2$ should not be confused with the subbundle $W$ complementary to $D$). A generator $Z_1$ of a gauge momentum
may be found using 
 the ansatz 
\begin{equation}
\label{E:Z1}
Z_1=g(\theta) W_1+ k(\theta)W_2
\end{equation}
and using  Theorem \ref{Th:skew-gauge-mom}
to determine the functions $g, k$.  According to the theorem, if $Z_1$ generates a gauge momentum then  
the functions $g, k$ should be such that 
\begin{equation}
\label{E:GaugeMomConds}
\begin{split}
\langle [Z_1, Y_2],Z_1\rangle&=\langle [Z_1,Y_2],Y_2\rangle=\langle [Z_1, Y_3],Y_3\rangle  =0, \\ \langle [Z_1, Y_3],Z_1\rangle &=0, \qquad  \langle [Z_1, Y_2],Y_3\rangle =-\langle [Z_1, Y_3],Y_2\rangle ,
\end{split}
\end{equation}
where  $\langle \cdot , \cdot \rangle$ is the kinetic energy Riemannian metric defined by the Lagrangian \eqref{E:LagSolidRev}, and $Y_2$ is any $G$-equivariant
 linear combination of $W_1$, $W_2$  (i.e. with  coefficients that are functions of $\theta$).

Independently of the choice made for $Y_2$,
a short direct calculation shows that the first three identities in \eqref{E:GaugeMomConds} hold for any functions $g, k$. On the other hand,
after a long calculation, it is seen that the other two are satisfied by any solution of the following system of linear ODE's:
\begin{equation}
\label{E:ODEBodyRev}
\frac{d}{d\theta}\left ( \begin{array}{c} g(\theta) \\ k(\theta) \end{array} \right )=L(\theta)\left ( \begin{array}{c} g(\theta) \\ k(\theta) \end{array} \right ),
\end{equation}
where the  $2\times 2$ matrix $L(\theta)=\frac{1}{K(\theta)} \tilde L(\theta)$ 
with
\begin{equation*}
K(\theta) := I_1I_3+ mI_1a_1(\theta)^2 + mI_3 f_2(\theta)^2,
\end{equation*}
and where the entries of $\tilde L$ are given by
\begin{equation}
\label{E:ExpM}
\begin{split}
\tilde L_{11}(\theta)&= mI_3f_2(\theta) \left ( \frac{R_m(\theta)-R_p(\theta)}{\sin \theta} \right )
 -ma_2(\theta) f_1(\theta)( I_3+mz(\theta)f_1(\theta)), \\
\tilde L_{12}(\theta)&= mI_3f_2(\theta)\cos \theta \left ( \frac{R_m(\theta)-R_p(\theta)}{\sin \theta} \right ) -m f_1(\theta)a_1(\theta) 
( I_3+mz(\theta)f_1(\theta)) , \\
\tilde L_{21}(\theta) &= mf_1(\theta)(I_1\sin^2\theta+I_3\cos^2\theta) \left ( \frac{R_p(\theta)-R_m(\theta)}{\sin \theta} \right ) 
  \\ & \qquad   +\frac{m a_2(\theta)}{ \sin^2\theta} \left ( ma_1(\theta) a_2(\theta) z(\theta) 
  +(R_m(\theta) -R_p(\theta))I_3 \cos \theta +(I_3-I_1) a_1(\theta) \sin \theta \cos \theta \right ),
  \\ \tilde L_{22}(\theta) &=m\cos \theta (I_1 a_1(\theta) \sin \theta +I_3f_2(\theta) \cos \theta) \left ( \frac{R_p(\theta)-R_m(\theta)}{\sin \theta} \right )
  \\ & \qquad   +mf_1(\theta)^2 (mz(\theta)a_2(\theta) + (I_3-I_1)\sin \theta \cos \theta).
\end{split}
\end{equation}
We point out that the function $K$ defined above is always positive and may be written as $K=I_1I_3 +m (\I \rho,  \rho)$ where, as in Section \ref{S:ChSphere}, $ (\cdot ,  \cdot)$ denotes
the standard scalar product in $\R^3$. As explained in \cite{BorMam}, the function $K$ is related to the density of an invariant measure for the system.

\begin{proposition}
\label{P:Existgk}
There exist two independent gauge momenta of the system on $D^*$.
\end{proposition}
\begin{proof}
The local existence of the two independent gauge momenta follows from the application of the existence theorem for ODE's to the system \eqref{E:ODEBodyRev}
to obtain two linearly independent solutions
(this can be done since the matrix $L(\theta)$ is smooth, see below).

In order to show that these integrals may be extended outside of our chart, we need to argue that  the generator $Z_1$ given by \eqref{E:Z1} admits an extension 
to all of $Q$.  This is certainly true if all the solutions to \eqref{E:ODEBodyRev} are
$2\pi$-periodic,  even functions of $\theta$.  To show that this is the case,
 we will prove that  the matrix $L(\theta)$ is smooth, odd, and $2\pi$-periodic (and therefore vanishes at $\theta=n\pi$, $n\in \mathbb{Z}$)
 and we will apply the following lemma.
  \begin{lemma}\label{lemma:Floquet}
Let $L(t)$ be an $n\times n$ matrix depending smoothly on $t$, $T$-periodic and odd (that is, $L(-t)=-L(t)$).
Then any solution to the differential equation $\dot \xx(t) = L(t)\xx(t)$ is even (i.e., $\xx(-t)=\xx(t)$) and $T$-periodic. 
\end{lemma}

\begin{proof}
Let $\xx_0\in\R^n$ and consider the initial value problem
$$\dot\xx(t) = L(t)\xx(t),\quad \xx(0)=\xx_0.$$
The evenness of $\xx(t)$ follows from uniqueness of solutions: if $\xx(t)$ were not even then it is easy to check that $\mathbf{y}(t):=\xx(-t)$ would be a different solution to the same initial value problem. 

Since $L$ is also $T$-periodic, it follows that 
$$L(T/2+t)=-L(-T/2-t) = -L(T/2-t)$$
and the argument above showing $\xx(t)$ is even also shows $\xx$ is ``even about $T/2$'': $\xx(T/2+t)=\xx(T/2-t)$. Consequently,
$$\xx(T/2+t)=\xx(T/2-t)=\xx(t-T/2)$$
whence $\xx$ is $T$-periodic.
\end{proof}

To show that   $L(\theta)$  has the required properties to apply the lemma, 
recall that $f_1$ and $f_2$ are even and $2\pi$-periodic. Consequently, the same is true about the functions
$z$, $R_p$, $R_m$ and $K$,
while $a_1$ and $a_2$ are odd and $2\pi$-periodic. Taking this into account, and in view of  \eqref{E:ExpM}, we conclude that $L$ is odd. Finally, the entries
of $L$ are seen to be smooth and vanish 
at $\theta=n\pi$ by using again \eqref{E:ExpM}. For this matter note that  these points correspond to the poles where, as a function of $\theta$, 
$R_p-R_m$ vanishes to  second order (since both $R_p$ and $R_m$ are even) and $a_1$ and $a_2$ also 
vanish.
 \end{proof}
 
 The differentials of the gauge momenta of the proposition above become dependent along the  points of  $D^*$ where the lifted $G$-action
 is not free. As will be explained below, these points correspond to a special kind of relative equilibrium.

The 3-form $\Lambda$ given by \eqref{E:Casimir3form}  that defines the desired gauge transformation can be computed without explicitly solving 
 \eqref{E:ODEBodyRev}. We will first find its expression on our chart and then give its expression on all of $Q$.
Let $(g(\theta), h(\theta))$,  be a  solution of \eqref{E:ODEBodyRev} with $h$ not identically zero. Then   $\{Z_1, Y_2, Y_3\}$ is an adapted basis 
at all points where $h(\theta)\neq 0$, where we have put $Y_2=W_1$. A simple calculation yields
\begin{equation*}
C_{123} = \langle [Z_1,Y_2],Y_3 \rangle =- mh(\theta)z(\theta)a_1(\theta).
\end{equation*}

We consider the dual basis $\{ \mu^\alpha \}$ of $\{Z_1, Y_2, Y_3\}$ given by
\begin{equation*}
\mu^1=\frac{1}{h(\theta)}\, d\psi, \qquad \mu^2=d\varphi - \frac{g(\theta)}{h(\theta)} \, d\psi, \qquad \mu^3= d\theta.
\end{equation*}
As for the Chaplygin sphere, this amounts to identifying $D^*=W^\circ$ where $W=\mathrm{span}\{\partial_x, \partial_y\}$. The  3-form $\Lambda$ defined
by \eqref{E:Casimir3form} is
unique since $r-\ell<3$ (item (ii) of Remark \ref{Rmk:ExistBasis}). It is given in our coordinates by
\begin{equation*}
\begin{split}
\Lambda=- mz(\theta) R_p(\theta) \sin \theta  \, d\varphi\wedge d\theta \wedge d\psi,
\end{split}
\end{equation*}
where we have written $a_1(\theta)= R_p(\theta) \sin \theta$. This 3-form may be written  in terms of the invariant 1-forms $\lambda^\alpha$ for
$SO(3)$ defined by \eqref{E:Def1formlambda} as
\begin{equation*}
\Lambda = mz(\gamma_3) R_p(\gamma_3)  \, \lambda^1 \wedge \lambda^2 \wedge \lambda^3.
\end{equation*}
Two facts about $\Lambda$ should be remarked at this point. Firstly, $\Lambda$
 admits a smooth extension to all of $Q$ including configurations where the action is not free. Secondly,
$\Lambda$ is independent of our choice of solution of the system  \eqref{E:ODEBodyRev}. Therefore, any gauge momentum of the system will be a Casimir function
of  the bracket that the corresponding bivector $\Pi_{nh}^\Lambda$ induces on the quotient space $D^*/G$.

The explicit expressions for the bracket  $\Pi_{nh}^\Lambda$ on $D^*$ may be obtained working with the basis of sections $\{W_1, W_2, Z_3\}$ and using the formulae in Section \ref{S:GaugeT}.
This is analogous to what was done in Section \ref{S:ChSphere}  for the Chaplygin ball. We do not give the details of this calculation. Instead we give  expressions for the brackets of the entries of $\gamma$ and of the angular momentum about the contact point
\begin{equation*}
M= \I\Omega + m\rho \times (\Omega \times \rho).
\end{equation*}
We have
\begin{equation}
\label{E:BrackSolidRev}
\begin{split}
& \{\gamma_i, \gamma_j\}=0,  \qquad \{M_i, \gamma_j \}=-\epsilon_{ijk}\gamma_k,  \\
 & \{M_i, M_j\} =  \epsilon_{ijk} \left ( - M_k +m R_m(\gamma_3) (\Omega, \gamma) \rho_k 
 +\frac{m(R_p(\gamma_3)-R_m(\gamma_3))z(\gamma_3)}{K(\gamma_3)} \left ( (M,\rho) \rho_k +T_k \right )\right ),
\end{split}
\end{equation}
where
\begin{equation*}
T_k=\frac{I_3(M_1\gamma_1+M_2\gamma_2)\gamma_k}{1-\gamma_3^2}, \; k=1,2, \qquad T_3=I_1M_3.
\end{equation*}
(Notice that the bracket has  no singularity at the poles $\gamma_3=\pm 1$ since $R_p-R_m$ vanishes there). In these expressions, we think that
$\Omega$ is written in terms of $M$ and $\gamma$ as
\begin{equation*}
\Omega=A M + m \frac{(AM,\rho)}{1-m(A\rho, \rho)}A\rho,
\end{equation*}
where $A:=(\I +m||\rho||^2)^{-1}$.

The reduction of the system by $G=SE(2)\times S^1$ can be performed in two steps since the individual actions of $SE(2)$ and $S^1$ commute.
Analogous to the case of the Chaplygin sphere,  the orbit space $D^*/SE(2)$ is smooth and  isomorphic to $S^2\times \R^3$, and points in this
space are labeled by the pair $\gamma, M$. The formulae \eqref{E:BrackSolidRev} can be interpreted as the
reduction of the bracket $\Pi_{nh}^\Lambda$ to  $D^*/SE(2)$. This partially reduced bracket does not satisfy the Jacobi identity except for very particular cases, like a
perfectly homogeneous sphere. In more general cases, like the so-called Routh sphere where the body is spherical but the centre of mass does not coincide with the
geometric centre, its characteristic distribution is not even integrable.

As indicated in \cite{BorMam}, the dynamics in this intermediately reduced space is given by 
\begin{equation*}
\dot M = M\times \Omega + m \dot \rho \times (\Omega \times \rho)+\frac{d V}{d \gamma_3} ( \gamma \times E_3), \qquad \dot \gamma = \gamma \times \Omega.
\end{equation*}
The  above equations are Hamiltonian with respect to the bracket \eqref{E:BrackSolidRev} and the Hamiltonian
$H(M,\gamma) = \frac{1}{2} (M,\Omega) + V(\gamma_3)$.

The ultimate reduction of the system is achieved by noticing that the action of $S^1$ on the orbit 
space $D^*/SE(2)$ is by simultaneous rotation on the planes $\gamma_1, \gamma_2$ and  $M_1, M_2$.
Note that the points having $\gamma_1=\gamma_2=0$ and $M_1=M_2=0$ are fixed by the action. The locus of these points corresponds to relative
equilibria where the body of revolution is steadily spinning about its axis of symmetry touching the plane at one of the poles. Along these points, the differentials of any two gauge momenta
of the system are dependent.

In order to see how the gauge momenta descend to  Casimir functions on the ultimately reduced space $D^*/G = (D^*/SE(2))/S^1$, we note that the
 linear functions on $D^*$  generated by  the vector fields $W_1$ and $W_2$ may be written as 
\begin{equation*}
p_{W_1} = (M,\gamma), \qquad p_{W_2}=M_3.
\end{equation*}
It follows that  the gauge momenta of the system are of the form
\begin{equation*}
C=g(\gamma_3)  (M,\gamma)+ k(\gamma_3) M_3,
\end{equation*}
where $g$ and $k$ are solutions to the system \eqref{E:ODEBodyRev}, expressed   as functions of $\gamma_3=\cos \theta$ (this is possible since the
solutions to this system are periodic and even functions of $\theta$ as shown in Proposition \ref{P:Existgk}). Using the differential equation satisfied by $g, k$ one can
 show that the Hamiltonian vector field of $C$ is
\begin{equation*}
X_C= k(\gamma_3) \left ( \gamma_2 \partial_{\gamma_1} - \gamma_1 \partial_{\gamma_2} + M_2 \partial_{M_1} -M_1 \partial_{M_2} \right ),
\end{equation*}
which is clearly vertical with respect to the action of $S^1$ defined above.

The ultimate reduced space $D^*/G=(D^*/SE(2))/S^1$ can be described by introducing 
generators of the ring of $S^1$-invariant polynomials on $D^*/SE(2)$. For example
\begin{equation*}
\begin{split}
\sigma_1&=\gamma_3, \qquad \sigma_2=\gamma_1M_2-\gamma_2M_1, \qquad \sigma_3=\gamma_1M_1
+\gamma_2M_2 \\
\sigma_4&=M_3, \qquad \sigma_5=M_1^2+M_2^2.
\end{split}
\end{equation*}
These functions identically satisfy
\begin{equation*}
\sigma_2^2+\sigma_3^2=\sigma_5 (1-\sigma_1^2),  \qquad \sigma_5 \geq 0.
\end{equation*}
The reduced space $D^*/G$ is then isomorphic to the four dimensional, semi-algebraic variety  $\mathcal{M}\subset \R^5$
defined by
\begin{equation*}
\mathcal{M}:=\left \{ \sigma \in \R^5 \, : \,   \sigma_2^2+\sigma_3^2=\sigma_5 (1-\sigma_1^2),  \quad \sigma_5 \geq 0 \right \}.
\end{equation*}
This space is not smooth having singularities along the two  lines
\begin{equation*}
L^{\pm}=\left \{  \sigma \in \R^5 \, : \, \sigma=\left ( \pm 1, 0 , 0 , \sigma_4, 0 \right )\, \right \}
\end{equation*}
that correspond to the  relative equilibria mentioned above. Each of these lines is a one dimensional stratum of $\mathcal{M}$.

By  $G$-invariance of  $\Pi_{nh}^\Lambda$, there is an induced  bracket $\{\cdot  ,\cdot \}_\mathcal{M}$ on $\mathcal{M}$ having 
\begin{equation*}
C_j(\sigma)=g_j(\sigma_1)\sigma_3 +(g_j(\sigma_1)\sigma_1+k_j(\sigma_1))\sigma_4, \qquad j=1,2,
\end{equation*}
as Casimir functions. Here $(g_j,k_j)$, $j=1,2$, are two linearly independent solutions of \eqref{E:ODEBodyRev} written as functions
of $\sigma_1=\gamma_3=\cos\theta$. These Casimir functions are   independent everywhere on $\mathcal{M}$ but their differentials are linearly dependent 
along the singular strata $L^{\pm}$. The Hamiltonian $H$ can be written in terms of $\sigma$ as
\begin{equation*}
H(\sigma)=\frac{1}{2}\left ( \frac{\sigma_5}{K_1(\sigma_1)}+\frac{\sigma_4^2}{K_3(\sigma_1)}  \right )  + \frac{m}{2} \frac{ \left ( \sigma_3 f_1(\sigma_1)K_3(\sigma_1) +\sigma_4 f_2(\sigma_1)K_1(\sigma_1) \right)^2}{K(\sigma_1)
K_1(\sigma_1)K_3(\sigma_1)}
\end{equation*}
where $K_j(\sigma_1):=I_j +m(1-\sigma_1^2)f_1(\sigma_1)^2+mf_2(\sigma_1)^2$, $j=1,3$. 
The ultimately reduced equations of motion can be formulated as
\begin{equation}
\label{E:Motion-Final-Reduced-space}
\dot \sigma_j = \{ \sigma_j , H\}_\mathcal{M}, \qquad j=1,\dots, 5.
\end{equation}
The crucial point of our construction is that the above equations are {\em true} Hamiltonian. Namely, the bracket $\{\cdot  ,\cdot \}_\mathcal{M}$ satisfies the
Jacobi identity. This follows immediately from Corollary \ref{C:Rank2}.

Equations (2.8) in  \cite{BorMam} give explicit expressions for the bracket $\chi(\sigma_1) \{\cdot  ,\cdot \}_\mathcal{M}$  (in terms of a different 
family of $S^1$-invariant functions on $D^*/SE(2)$), where the function $\chi(\sigma_1)=\sqrt{\frac{K_1(\sigma_1)}{1-\sigma_1^2}}$.
As was already pointed out in \cite{Ramos},  it is not necessary to introduce this
conformal factor in order to satisfy the Jacobi identity. There is no explanation in \cite{BorMam} about the origin of this bracket, and it is likely that it was found by the authors
using an ad hoc approach.

The restriction of the system to the 2-dimensional symplectic leaf determined as the level set of the Casimir functions, defines a one degree of freedom, and hence integrable, Hamiltonian
system. Apparently (see  \cite{BorMam}) the reduction of the  integration of the reduced system to a set of 2 linear ODE's was known to   Chaplygin.
More details about the explicit integration of the reduced system can 
 be found in  \cite{BorMam} or \cite{CDS-book}. 
We simply mention that the generic solutions are periodic.

We stress that our approach for the reduction of the system  follows the philosophy and treatment in  \cite{CDS-book} but with a fundamental difference: we are
performing the Poisson reduction of the system with respect to the bracket $\Pi_{nh}^\Lambda$ and {\em not} with respect to $\Pi_{nh}$. 
By following the reduction of $\Pi_{nh}$, the authors of \cite{CDS-book} arrive to an equation ((180) in their text) analogous to 
 \eqref{E:Motion-Final-Reduced-space} but with respect to a bracket of functions that does not satisfy the Jacobi identity.
The authors do seem to notice that for fixed values of the integrals $C_1, C_2$ one has a one-degree of freedom Hamiltonian system (section 6.3.7.4),
but they do not provide a link between this observation and their reduction.

\appendix
\section{Proof of Lemma~\ref{L:Global3form} }

The proof of Theorem~\ref{th:main} depends in a crucial way on the existence of the 3-form $\Lambda$ that satisfies the conditions 
of Lemma~\ref{L:Global3form}. Here we present the a proof  that such a $\Lambda$ always exists. As explained at several points of the
text, we will only prove that $\Lambda$ is well defined on the open dense subset $Q_f\subset Q$ where the $G$-action is free.

Recall  the assumption made at the beginning of Section~\ref{SS:3form}: the  chosen subbundle $W$ of $TQ$ with the property 
$TQ=D\oplus W$ is $G$-invariant.

\subsection{Local considerations}

Consider an adapted basis   $\{X_\alpha\}$  of sections of $D$. Recall that  this means that all 
vector fields $X_\alpha$ are equivariant sections of $D$ and that
 $X_b$ are gauge momentum generators  for $b=1,\dots,\ell$.  That is, $X_b$ is a section of the distribution $\mathcal{S}$ 
 on $Q$ defined point-wise as $\mathcal{S}_q = D_q \cap( \g\cdot q)$, and satisfies $\dot p_{X_b}=0$.
 In virtue of  Lemma~\ref{L:ExistBasis}, we may assume that our adapted basis of sections is defined on a
 $G$-invariant open subset $U\subset Q_f$.

  In the above paragraph, as in the remainder of this section, we do not distinguish $Z_b$ and $Y_J$ as in Definition \ref{E:DefAdaptedBasis}, as there would be too much notation---we rely on the indices to distinguish the type of generator. Recall that our convention is $b,c\dots\,\in\{1,\dots,\ell\}$ (corresponding to gauge momentum generators), $I,J,\dots\,\in\{\ell+1,\dots,r\}$ and $\alpha,\beta,\dots\,\in\{1,\dots,r\}$, where $r$ is the rank of $D$.

A 3-form $\Lambda$ satisfying conditions (i) and (ii) of Lemma~\ref{L:Global3form} may be defined  on $U$ by the formula
\begin{equation} 
\label{eq:appendix-3-form}
\Lambda = \tfrac16 B_{\alpha\beta\gamma}\,\mu^\alpha\wedge\mu^\beta\wedge\mu^\gamma,
\end{equation}
where  the coefficients $B_{\alpha\beta\gamma}$  are any $G$-invariant functions, alternating in the indices and satisfying
\begin{equation}\label{eq:appendix-B-coefficients}
B_{b\beta\gamma} = \left<[X_b,X_\beta],\,X_\gamma\right>.
\end{equation}
Here $B_{IJK}$ are arbitrary smooth $G$-invariant functions, alternating in the indices. Note that  
the skew-symmetry condition
$B_{b\beta\gamma}=-B_{b\gamma\beta}$ holds in virtue of Theorem~\ref{Th:skew-gauge-mom}.
 As usual,   $\{\mu^\alpha\}$ in \eqref{eq:appendix-3-form} is the dual basis of $\{X_\alpha\}$. Namely, they 
  are 1-forms defined on $U$ that annihilate $W$ and satisfy $\mu^\alpha(X_\beta)=
\delta_\beta^\alpha$. Moreover, they are  $G$-invariant by $G$-invariance of $W$, and hence $\Lambda$ is also $G$-invariant.

Now let $V$ be a possibly different $G$-invariant open subset of $Q_f$, and let $\{Y_\alpha\}$ be a (new) adapted basis of 
$D$ on $V$, with $Y_\alpha=M_\alpha^\beta X_\beta$ on the intersection $U\cap V$. In order to be gauge momentum generators, we require $Y_c = M_c^bX_b$, with $M_c^b\in\R$ (constants), and in order that the $Y_\beta$ are equivariant we require that all the coefficients $M_\beta^\alpha$ be $G$-invariant functions.  

(Here we are assuming the distribution is strongly nonholonomic on the configuration space, otherwise the $M_c^b$ are only annihilated by every vector field tangent to the distribution $D$; the proof in the more general case proceeds in the same way.)  

The dual basis $\{\mu^\alpha\}$ transforms into a new basis $\{ \nu^\beta \}$ dual to the $Y_\beta$ satisfying $\nu^\beta=\Mbar^\beta_\alpha\mu^\alpha$, where $\bar{M}$ is the inverse matrix of $M$; that is, $\Mbar_\alpha^\beta M_\beta^\gamma = \delta_\alpha^\gamma$ (Kronecker $\delta$). Let us emphasise that 
\begin{equation}\label{eq:M is upper triangular}
M_b^I=\bar M_b^I=0,\quad\forall b,I.
\end{equation}
We define 
$$\Lambda' = \tfrac16 B'_{\alpha\beta\gamma}\,\nu^\alpha\wedge\nu^\beta\wedge\nu^\gamma,
$$
where the $G$-invariant alternating  coefficients $B'_{\alpha \beta\gamma}$ are defined as for $\Lambda$. Namely, 
 $B'_{b\beta\gamma}$ is given by \eqref{eq:appendix-B-coefficients}, using the $Y_\alpha$ in place of the $X_\alpha$, and 
  $B'_{IJK}$ are arbitrary.

We wish to compare $\Lambda$ and $\Lambda'$ on the intersection $U\cap V$. 

\begin{lemma}\label{lemma:B-expand} The coefficients $B_{\alpha\beta\gamma}$ and $B'_{\alpha\beta\gamma}$ are related by
\begin{eqnarray}
B'_{b\beta\gamma} &=& M_b^\rho M_\beta^\sigma M_\gamma^\tau \,B_{\rho\sigma\tau}  
\label{eq:B'_a}\\ 
B'_{IJK} &=& M_I^\rho M_J^\sigma M_K^\tau \,B_{\rho\sigma\tau} +E_{IJK}, \label{eq:B'_I}
\end{eqnarray}
for some $G$-invariant functions $E_{IJK}$ on $Q$, alternating in the indices.
The forms themselves are related more simply by,
\begin{equation}
\label{eq:nu-mu}
\nu^\alpha\wedge\nu^\beta\wedge\nu^\gamma = \Mbar^\alpha_\delta\Mbar^\beta_\eps\Mbar^\gamma_\eta\, \mu^\delta\wedge\mu^\eps\wedge\mu^\eta.
\end{equation}
\end{lemma}

Note that since $M_b^I=0$  \eqref{eq:M is upper triangular}, all the non-zero terms on the right-hand side of \eqref{eq:B'_a} involve only  coefficients of the form $B_{b\sigma\tau}$, and not the $B_{IJK}$.

\begin{proof} The expression \eqref{eq:nu-mu} follows immediately from the relation $\nu^\alpha=\Mbar^\alpha_\beta\mu^\beta$ given above.

For the $B'$ coefficients, expand $B'$ in terms of the $X_\alpha$:
\begin{eqnarray*}
B'_{b\beta\gamma} &=& \left<[Y_b,Y_\beta],\,Y_\gamma\right> \\
&=& \left<[M_b^cX_c,M_\beta^\sigma X_\sigma],\,M_\gamma^\tau X_\tau\right> \\
&=& M_b^cM_\beta^\sigma M_\gamma^\tau B_{b\sigma\tau}
       + M_b^c M_\gamma^\tau X_c(M_\beta^\sigma)\left<X_\sigma,X_\tau\right> 
       -  M_\beta^\sigma M_\gamma^\tau X_\sigma(M_b^c)\left<X_c,X_\tau\right>.
\end{eqnarray*}
However, the final two terms both vanish because firstly $X_c$ is tangent to the group orbit and $M_\beta^\sigma$ is invariant (so constant on group orbits), and secondly $M_b^c$ is constant.  Thus, $B'_{b\alpha\beta} = M_b^cM_\alpha^\sigma M_\beta^\tau B_{c\sigma\tau}$.  The first equation \eqref{eq:B'_a} then follows from \eqref{eq:M is upper triangular}. 

Equation \eqref{eq:B'_I} can be taken as a definition of $E_{IJK}$, where the invariance and the alternating structure is clear. 
\end{proof}

Recall that for each $q$ the subspace $\mathcal{S}_q\subset T_qQ$ is defined to be $\mathcal{S}_q=D_q\cap \g\cdot q$. 
We also write $\mathcal{S}_0\subset\mathcal{S}$ to be the sub-distribution spanned by the generators of the gauge momenta, which is therefore assumed to be of constant rank.   Recall that $W$ is the distribution complementary to $D$ on which all $\mu^\alpha$ vanish.

\begin{proposition}\label{prop:transformation of Lambda}
On $U\cap V$, the difference $\Psi:=\Lambda-\Lambda'$ is an invariant 3-form annihilating $\mathcal{S}_0\oplus W$, so is of the form 
$$\Psi = E_{IJK}\,\mu^I\wedge\mu^J\wedge \mu^K.$$
\end{proposition}

\begin{proof}
Using the lemma above,
\begin{eqnarray*}
\Lambda' &=& \tfrac16 B'_{\alpha\beta\gamma}\,\nu^\alpha\wedge\nu^\beta\wedge\nu^\gamma \\
&=& \tfrac16 B'_{\alpha\beta\gamma}\,\Mbar^\alpha_\delta\Mbar^\beta_\eps\Mbar^\gamma_\eta\, \mu^\delta\wedge\mu^\eps\wedge\mu^\eta\\
&=& \tfrac16  M_\alpha^\rho M_\beta^\sigma M_\gamma^\tau \,B_{\rho\sigma\tau} \,\Mbar^\alpha_\delta\Mbar^\beta_\eps\Mbar^\gamma_\eta\, \mu^\delta\wedge\mu^\eps\wedge\mu^\eta 
    + E_{IJK}\Mbar^I_\delta\Mbar^J_\eps\Mbar^K_\eta\, \mu^\delta\wedge\mu^\eps\wedge\mu^\eta \\
&=& \tfrac16 \delta_\delta^\rho \delta_\eps^\sigma \delta_\eta^\tau B_{\rho\sigma\tau} \, \mu^\delta\wedge\mu^\eps\wedge\mu^\eta 
    + E_{IJK}\Mbar^I_R\Mbar^J_S\Mbar^K_T\, \mu^R\wedge\mu^S\wedge\mu^T \\
&=& \tfrac16  B_{\delta\eps\eta}\, \mu^\delta\wedge\mu^\eps\wedge\mu^\eta 
    + E_{IJK}\Mbar^I_R\Mbar^J_S\Mbar^K_T\, \mu^R\wedge\mu^S\wedge\mu^T \\
&=& \Lambda + E'_{IJK}\, \mu^I\wedge\mu^J\wedge\mu^K
\end{eqnarray*}
as required (the final step involves relabelling $RST$ to $IJK$).

\end{proof}

\subsection{Global considerations} As was pointed out in Remark~\ref{Rmk:ExistBasis}, if $k-\ell<3$ then $\Lambda=\Lambda'$ and the 3-form is uniquely defined.  There remains the question of whether when $k-\ell\geq 3$, the locally-defined 3-forms $\Lambda$ can be chosen to agree everywhere, by choosing the $B_{IJK}$ suitably.  The answer is yes, by a standard  partition of unity argument used in \v{C}ech cohomology, as follows.

\begin{proposition}\label{prop:global consistency}
Let $\mathfrak{U}=\{U_i\}$ be a cover of $Q_f$ by $G$-invariant open sets, and on each $U_i$ suppose a 3-form $\Lambda_i$ is selected, of the form \eqref{eq:appendix-3-form}.  There exist $G$-invariant 3-forms $\Psi_i$ on $U_i$ annihilating $\mathcal{S}_0\oplus W$,  such that the forms 
$$\widetilde{\Lambda}_i :=\Lambda_i + \Psi_i$$
define a global 3-form on $Q$; that is on each intersection $U_i\cap U_j$, $\widetilde{\Lambda}_i = \widetilde{\Lambda}_j$.
\end{proposition}

\begin{proof}
Let $\{\phi_i\}$ be a partition of unity subordinate to the cover $\mathfrak{U}$, by $G$-invariant functions.  On each $U_i\cap U_j$ let $\Psi_{ij}=\Lambda_i - \Lambda_j$. By Proposition\,\ref{prop:transformation of Lambda}, the $\Psi_{ij}$ are invariant forms annihilating $\mathcal{S}_0\oplus W$, and they clearly satisfy the cocycle condition,
$$\Psi_{ij}+\Psi_{jk}+\Psi_{ki}=0,\quad\text{wherever } U_i\cap U_j\cap U_k\neq\emptyset.$$ 
Now, for each $i$ define
$$\Psi_i = \sum_k \phi_k \Psi_{ik}.$$
(Note that $\Psi_{ii}=0$, and the sum is over all $k$.) Then with $\widetilde{\Lambda}_i=\Lambda_i+\Psi_i$,
\begin{eqnarray*}
\widetilde{\Lambda}_i- \widetilde{\Lambda}_j &=& \Lambda_i-\Lambda_j + \Psi_i-\Psi_j \\ 
&=& \Psi_{ij} + \sum_k\phi_k(\Psi_{ik}-\Psi_{jk}) \\
 &=& \Psi_{ij} + \sum_k\phi_k \Psi_{ji} \\ 
 &=& \Psi_{ij} + \Psi_{ji} \ = \ 0,
\end{eqnarray*}
as required.
\end{proof}

\noindent \textbf{Funding:} This research was made possible by a Newton Advanced Fellowship from the Royal Society, ref: 
NA140017.

\noindent \textbf{Conflict of Interest:} The authors declare that they have no conflict of interest.

\end{document}